\documentclass{amsart}

\usepackage{amssymb}
\usepackage{amsmath}
\usepackage{amsthm}
\usepackage{graphicx}

\newtheorem{theorem}{Theorem}[section]
\newtheorem{lemma}[theorem]{Lemma}
\newtheorem{proposition}[theorem]{Proposition}
\newtheorem{corollary}[theorem]{Corollary}

\newtheorem{_definition}[theorem]{Definition}

\newtheorem{_remark}[theorem]{\it Remark}
\newenvironment{remark}{\begin{_remark}\rm}{\end{_remark}}

\newtheorem{_example}[theorem]{Example}
\newenvironment{example}{\begin{_example}\rm}{\end{_example}}

\numberwithin{equation}{section}
\numberwithin{table}{section}
\numberwithin{figure}{section}

\newcommand{\F}{\mathord{\mathbb F}}
\renewcommand{\P}{\mathord{\mathbb  P}}
\newcommand{\C}{\mathord{\mathbb  C}}
\newcommand{\Q}{\mathord{\mathbb  Q}}

\newcommand{\Z}{\mathord{\mathbb Z}}

\newcommand{\CCC}{\mathord{\mathcal C}}
\newcommand{\EEE}{\mathord{\mathcal E}}
\newcommand{\FFF}{\mathord{\mathcal F}}
\newcommand{\HHH}{\mathord{\mathcal H}}
\newcommand{\III}{\mathord{\mathcal I}}
\newcommand{\KKK}{\mathord{\mathcal K}}
\newcommand{\LLL}{\mathord{\mathcal L}}
\newcommand{\MMM}{\mathord{\mathcal M}}
\newcommand{\NNN}{\mathord{\mathcal N}}
\newcommand{\OOO}{\mathord{\mathcal O}}
\newcommand{\PPP}{\mathord{\mathcal P}}

\newcommand{\TTT}{\mathord{\mathcal T}}
\newcommand{\UUU}{\mathord{\mathcal U}}

\newcommand{\XXX}{\mathord{\mathcal X}}

\newcommand{\SSSS}{\mathord{\mathfrak S}}

\newcommand{\inj}{\hookrightarrow}
\newcommand{\surj}{\mathbin{\to \hskip -7pt \to}}
\newcommand{\isom}{\mathbin{\,\raise -.6pt\rlap{$\to$}\raise 3.5pt \hbox{\hskip .3pt$\mathord{\sim}$}\,}}

\newcommand{\ratmap}{\cdot\hskip -1pt\cdot\hskip -1.5pt\to}

\newcommand{\set}[2]{\{\; {#1} \; \mid \; {#2} \;  \}}
\newcommand{\shortset}[2]{\{{#1} \mid  {#2} \}}

\newcommand{\gen}[1]{\langle {#1}  \rangle}
\newcommand{\genset}[2]{\langle\; {#1} \; \mid  \; {#2}\; \rangle}

\newcommand{\wt}{\widetilde}

\newcommand{\tensor}{\otimes}

\newcommand{\sprime}{\sp\prime}

\newcommand{\spar}[1]{\sp{(#1)}}

\newcommand{\sperp}{\sp{\perp}}

\newcommand{\dual}{\sp{\vee}}

\newcommand{\inv}{\sp{-1}}

\newcommand{\erase}[1]{}

\newcommand{\Hom}{\mathord{\mathrm {Hom}}}
\newcommand{\sheafHom}{\mathord{\mathcal Hom}}

\newcommand{\GL}{\mathord{\mathit {GL}}}

\newcommand{\id}{\mathord{\mathrm {id}}}
\newcommand{\pr}{\mathord{\mathrm {pr}}}

\newcommand{\Spec}{\operatorname{\mathrm {Spec}}}

\newcommand{\Ker}{\operatorname{\mathrm {Ker}}}
\newcommand{\rank}{\operatorname{\mathrm {rank}}}
\newcommand{\disc}{\operatorname{\mathrm {disc}}}

\newcommand{\ang}[1]{\langle #1\rangle}

\newcommand{\rmand}{\textrm{and}}

\newcommand{\qquand}{\qquad\rmand\qquad}
\newcommand{\quand}{\quad\rmand\quad}
\newcommand{\qand}{\;\;\rmand\;\;}
\newcommand{\sqand}{\;\rmand\;}
\newcommand{\mystruthd}[2]{\phantom{\hbox{\vrule  height #1 depth #2}}}

\newcommand{\Gr}{G}
\newcommand{\GGlc}{{\Gr_{n, l}}\times{\Gr_n^c}}
\newcommand{\Fr}{\phi}

\newcommand{\Nlc}{N_{l}^c}
\newcommand{\sparq}{\spar{q}}
\newcommand{\Incvar}{\III}

\newcommand{\cond}{{\rm C}}
\newcommand{\statement}{{\rm S}}
\newcommand{\latN}{\hskip .2pt\NNN}
\newcommand{\latNtl}{\hskip .2pt\wt{\NNN}}
\newcommand{\Nprim}{\latN_{\mathord{\rm{prim}}}}
\newcommand{\NSigma}{\latN_{\Sigma}}

\newcommand{\Nmin}{N_{\mathord{\rm{min}}}}
\newcommand{\Pset}[1]{\PPP_{#1}}
\newcommand{\Lset}[1]{\LLL_{#1}}

\newcommand{\weight}{\mathord{\rm{wt}}}

\newcommand{\baseptP}{B_0}

\newcommand{\intdim}{m}
\newcommand{\sumcodim}{k}
\newcommand{\symchi}{\psi}
\newcommand{\tlchi}{\tilde{\psi}}

\newcommand{\Tq}{T}

\begin{document}
\title[Frobenius incidence varieties]%
{On Frobenius incidence varieties of linear subspaces over finite fields}

\author{I.~Shimada}
\email{shimada@math.sci.hiroshima-u.ac.jp
}

\address{%
Department of Mathematics, 
Graduate School of Science, 
Hiroshima University,
1-3-1 Kagamiyama, 
Higashi-Hiroshima, 
739-8526 JAPAN
}

\thanks{Partially supported by
 JSPS Core-to-Core Program No.~18005  and 
 JSPS Grants-in-Aid for Scientific Research (B) No.~20340002.
}

\subjclass[2000]{14M15,  14G15,  14C25}



\begin{abstract} 
We  define  Frobenius incidence varieties
by means of the  incidence relation of Frobenius images of 
linear subspaces in a fixed vector space over a finite field,
and investigate their  properties
such as supersingularity, Betti numbers and  unirationality.
These varieties are variants of the Deligne-Lusztig varieties.
We then 
study  the  lattices associated with algebraic cycles on them.
We obtain a positive-definite lattice of rank $84$
that yields a dense sphere packing 
from a $4$-dimensional Frobenius incidence variety 
in characteristic $2$.
 \end{abstract}

\maketitle


\section{Introduction}
Codes arising from the rational points of Deligne-Lusztig varieties have been 
studied in several cases~\cite{MR1186416, MR1866342,MR1759842}.
In this paper,
we investigate lattices arising from algebraic cycles 
on certain variants of Deligne-Lusztig varieties,
which we call \emph{Frobenius incidence varieties}.
We study basic  properties of  Frobenius incidence varieties
such as supersingularity, Betti numbers and unirationality.
By means of intersection pairing of algebraic cycles
on a $4$-dimensional Frobenius incidence variety over $\F_2$,
we obtain a positive-definite lattice of rank $84$
that yields a dense sphere packing.
\subsection{An illustrating example}
Before giving the general definition of Frobenius incidence varieties in~\S\ref{sec:mainresults},
we present the simplest example of  Frobenius incidence \emph{surfaces},
hoping that it explains the motivation for the main results of this paper.
\par
\medskip
We fix a  vector space $V$ over $\F_p$ of dimension $3$
with coordinates $(x_1, x_2, x_3)$,
and consider the projective plane $\P_* (V)$
with the homogeneous coordinate system $(x_1: x_2: x_3)$.
Let $\bar{F}$ be an algebraic closure of $\F_p$.
An $\bar{F}$-valued point  $(a_1: a_2: a_3)$ of $\P_* (V)$
corresponds to the  $1$-dimensional linear subspace of $V\tensor \bar{F}$ spanned by 
$(a_1, a_2, a_3) \in V\tensor \bar{F}$.
Let $\P^*(V)$ denote the dual projective plane with homogeneous coordinates 
$(y_1: y_2: y_3)$ dual to $(x_1: x_2: x_3)$.
An $\bar{F}$-valued point  $(b_1: b_2: b_3)$ of $\P^* (V)$
corresponds to the  $2$-dimensional linear subspace of $V\tensor \bar{F}$ defined by 
$b_1x_1+b_2x_2+b_3x_3=0$.
The incidence variety is a hypersurface of $\P_*(V)\times \P^*(V)$
defined by
$x_1y_1+x_2 y_2+x_3 y_3=0$,
which parametrizes all the pairs $(L, M)$ of a $1$-dimensional linear subspace $L$
and a $2$-dimensional linear subspace $M$ such that $L\subset M$.
\par
\medskip
Let $q$ be a power of $p$ by a positive integer.
The $q$th power Frobenius morphism of $V\tensor \bar{F}$
is the morphism from  $V\tensor \bar{F}$ to itself given by 
$(x_1, x_2, x_3)\mapsto (x_1^q, x_2^q, x_3^q)$.
For a linear subspace $N$ of $V\tensor \bar{F}$,
we denote by $N^q\subset V\tensor \bar{F}$ the image of $N$ by the $q$th power Frobenius morphism,
which is again a  linear subspace  of $V\tensor \bar{F}$. 
If a $2$-dimensional linear subspace $M$ of 
$V\tensor \bar{F}$ corresponds to a point $(b_1: b_2: b_3)\in \P^* (V)$,
the linear subspace $M^q$ 
corresponds to the point $(b_1^q: b_2^q: b_3^q)$.
\par
\medskip
We take two Frobenius twists of the incidence variety,
and take their intersection.
Let $r$ and $s$ be powers of $p$ by positive integers.
The hypersurface of $\P_*(V)\times \P^*(V)$
defined by
\begin{equation}\label{eq:r1}
x_1^r y_1+x_2^r y_2+x_3^r y_3=0
\end{equation}
parametrizes  the pairs $(L, M)$  such that
$L^r\subset M$, while 
the hypersurface of $\P_*(V)\times \P^*(V)$
defined by
\begin{equation}\label{eq:1s}
x_1 y_1^s+x_2 y_2^s+x_3 y_3^s=0
\end{equation}
parametrizes  the pairs $(L, M)$ such that
$L\subset M^s$.
Using affine coordinates of  $\P_*(V)\times \P^*(V)$,
we see that these two hypersurfaces~\eqref{eq:r1} and~\eqref{eq:1s} intersect transversely.
Let $X$ be the intersection,
which is a smooth surface parameterizing 
the pairs $(L, M)$ such that
$$
L^r\subset M\quad\textrm{and}\quad L\subset M^s,
$$
or equivalently
$$
L^r\subset M\cap M^{rs},
$$
or equivalently
$$ 
L+L^{rs}\subset M^s.
$$
\par
\medskip
We put $q:=rs$, and 
count the $\F_{q^\nu}$-rational points of the surface $X$
for positive integers $\nu$,
that is, we count the number of the pairs $(L, M)$  of $\F_{q^\nu}$-rational linear subspaces
$L$ and $M$
that satify the above conditions.
Consider 
the first projection $\pi_1: X\to \P_*(V)$.
Let $P$ be an $\F_{q^\nu}$-rational  point  of $\P_*(V)$  corresponding to $L\subset V\tensor \bar{F}$.
Then, if $\dim (L+L^{q})=2$,  the fiber $\pi_1\inv (P)$
consists of  a single $\F_{q^\nu}$-rational point  
corresponding to the  $\F_{q^\nu}$-rational subspace $M$
such that $L+L^{q}= M^s$,
while, if $\dim (L+L^{q})=1$,  it is isomorphic to an $\F_{q^\nu}$-rational projective line 
parameterizing   subspaces $M$
such that $L+L^{q}\subset  M^s$.
Since $\dim (L+L^{q})=1$ holds if and only $P$ is an $\F_{q}$-rational point of 
$\P_*(V)$, 
the number of the $\F_{q^\nu}$-rational points of $X$ is equal to
$$
\left(\frac{q^{3\nu}-1}{q^\nu-1}-\frac{q^3-1}{q-1}\right)+
\left(\frac{q^{2\nu}-1}{q^\nu-1}\right)\cdot \left(\frac{q^3-1}{q-1}\right). 
$$
If we put 
\begin{equation*}
N(t):=t^2 +(q^2+q+2)t+1,
\end{equation*}
then this number is equal to $N(q^\nu)$.
In particular, from the classical theorems
on the Weil conjecture~(see,~for example,~\cite[App.~C]{MR0463157}),  
we obtain the Betti  numbers $b_i(X)$ of the surface $X$.
We have   $b_0(X)=b_4(X)=1$, $b_1(X)=b_3(X)=0$ and
$$
b_2(X)=q^2+q+2.
$$
\par
Remark that, when $r>2$ and $s>2$, 
the canonical line bundle $\OOO(r-2, s-2)$ of $X$ is ample and has non-zero global sections.
Hence,
the complex algebraic surface $X_{\C}$ 
defined by~\eqref{eq:r1} and~\eqref{eq:1s} in $\C\P^2\times\C\P^2$
cannot be unirational (see~\cite[Chap.~V, Remark~6.2.1]{MR0463157}), and 
the second Betti cohomology group of $X_{\C}$
cannot be spanned by the classes of algebraic
cycles because of  the Hodge-theoretic reason
(see~\cite[p.~163]{MR1288523}).
\par
\medskip
However,
the surface $X$ has the so-called pathological properties of algebraic varieties in positive characteristics,
that is, $X$ contradicts naive expectations from the properties of $X_{\C}$.
Since the projection $\pi_1: X\to \P_*(V)$ gives rise to a purely inseparable extension of the function fields,
$X$ is unirational.
Moreover, 
since $N(t)$ is a polynomial in $t$,
the eigenvalues of the $q$th power Frobenius endmorphism 
acting on the $l$-adic cohomology ring of $X$
is a power of $q$ by integers.
According to the Tate conjecture,
the second $l$-adic cohomology group of $X$ should be spanned 
by the classes of algebraic curves on $X$  defined over $\F_{q}$.
This is indeed the case.
There are $2(b_2(X)-1)$ special rational curves defined over $\F_{q}$ on $X$;
the fibers $\Sigma_P$ of  $\pi_1: X\to \P_*(V)$ over the $\F_{q}$-rational 
points $P$  of $\P_*(V)$,
and the fibers $\Sigma\sprime_Q$ of  $\pi_2: X\to \P^*(V)$ over the $\F_{q}$-rational 
points $Q$  of $\P^*(V)$.
By calculating the intersection numbers between these curves (see~\cite[Chap.~V, \S1]{MR0463157}), 
we see that the numerical equivalence classes of 
$\Sigma_P$ and $\Sigma\sprime_Q$ together with the classes of the line bundles
$\OOO(1, 0)$ and $\OOO(0, 1)$
form a hyperbolic lattice $\NNN(X)$ of rank $b_2(X)$ under the intersection pairing.
Thus their classes span the second $l$-adic cohomology group of $X$.
\par
\medskip
When $p=r=s=2$, 
the surface $X$ is a supersingular $K3$ surface 
in characteristic $2$ with $|\disc \NNN(X)|=4$.
The defining  equations~\eqref{eq:r1} and~\eqref{eq:1s},
which were discovered by Mukai,
and the configuration of the $21+21$ rational curves $\Sigma_P$ and  $\Sigma\sprime_Q$
played an important role in the study of the automorphism group
of this $K3$ surface in Dolgachev-Kondo~\cite{MR1935564}. 
\par
\medskip
Looking at this example, 
we expect that  the lattice $\NNN(X)$ possesses  interesting properties.
In particular,  its primitive part can yield a dense sphere packing.
\subsection{Definitions and the main results}\label{sec:mainresults}
We give the  definition of Frobenius incidence varieties,
 and state the main results of this paper.
\par
\medskip
Let $p$ be a prime,
and let $q:=p^\nu$ be a power of $p$ by a positive integer $\nu$.
For a field $F$ of characteristic $p$
with an algebraic closure $\bar{F}$, we put
$$
F\sp{q}:=\shortset{x^q}{x\in F}
\quand
 F\sp{1/q}:=\shortset{x\in \bar{F}}{x^q\in F}.
$$
For a scheme $Y$ defined over a subfield of $F$, 
we denote by $Y(F)$ the set of $F$-valued points of $Y$.
\par
\medskip
We fix an $n$-dimensional linear space $V$
over $\F_p$ with $n\ge 3$,
and denote by $\Gr_{n,l}=\Gr_n^{n-l}$ or
by $\Gr_{V,l}=\Gr_V^{n-l}$ 
the Grassmannian variety of  $l$-dimensional subspaces of $V$.
To ease the notation,
we use the same letter $L$ to denote an $F$-valued  point 
$L\in {\Gr}_{n,l} (F)$  of ${\Gr}_{n,l}$
and the corresponding linear subspace $L\subset {V_F}:=V\tensor F$.
Moreover,
for an extension field $F\sprime$ of $F$, we  write $L$
for the  linear subspace $L\tensor_{F} F\sprime$ of $V_{F\sprime}$.
Let $\Fr$ denote the $p$\,th power Frobenius morphism of ${\Gr}_{n,l}\tensor \bar{\F}_p$
over $\bar{\F}_p$, 
and let $\Fr\sparq$  be  the $\nu$-fold iteration of $\Fr$.
Then $\Fr\sparq$
 induces a bijection   from ${\Gr}_{n,l} (F)$ to ${\Gr}_{n,l} (F\sp{q})$.
We denote by
$L^q\in {\Gr}_{n,l} (F\sp{q})$ the image of $L\in {\Gr}_{n,l} (F)$ by $\Fr\sparq$,
and by $L\sp{1/q}\in {\Gr}_{n,l} (F\sp{1/q})$ the point that is mapped to $L$ by $\Fr\sparq$.
Let $(x_1, \dots, x_n)$ be $\F_p$-rational coordinates of $V$.
If $L$ is defined in $V_F$ by linear equations $\sum_j a_{ij}x_j=0$ $(i=1, \dots, n-l)$ with $a_{ij}\in F$,
then $L^q$ is defined by the linear equations $\sum_j a_{ij}^qx_j=0$ $(i=1, \dots, n-l)$.
\par
\medskip
Let $l$ and $c$ be positive integers such that $l+c<n$.
We denote by $\Incvar$ the incidence subvariety of ${\Gr}_{n,l}\times {{\Gr}_n^c}$.
By definition, 
$\Incvar$ is 
the reduced subscheme of ${\Gr}_{n,l}\times {{\Gr}_n^c}$
such that, for any field $F$ of characteristic $p$, we have
$$
\Incvar(F)=\set{(L, M)\in {\Gr}_{n,l}(F)\times {{{\Gr}_n^c}}(F)}{L\subset M}.
$$
\par
Let $r$ and $s$ be powers of $p$ by non-negative integers
such that 
$r>1$ or $s>1$ holds.
We define 
the \emph{Frobenius incidence variety}  $X[r, s]_l^c$ to be  the scheme-theoretic intersection
of the pull-backs 
$(\Fr\spar{r} \times \id )^*\,\Incvar$ and $( \id \times \Fr \spar{s})^*\,\Incvar$
of $\Incvar$,
where $\id$ and $\Fr\spar{1}$ denote the identity map:
\begin{equation*}\label{eq:schemeX}
X[r, s]_l^c:=\,(\Fr\spar{r} \times \id )^*\,\Incvar\,\cap\, ( \id \times  \Fr\spar{s})^*\,\Incvar.
\end{equation*}
For simplicity,
we write $X$  or $X[r, s]$ or $X_l^c$ 
for $X[r, s]_l^c$ if there is no possibility of confusion.
The scheme $X$ is  defined over $\F_p$ and, for any field $F$ over $\F_p$, 
we have 
\begin{equation}\label{eq:setX}
X(F)=
\set{(L, M)\in {\Gr}_{n,l}(F)\times {{{\Gr}_n^c}}(F)}%
{L\sp{r}\subset M \qand L\subset M\sp{s}}.
\end{equation}
We have the following:
\begin{proposition}\label{prop:smoothdim}
The projective scheme $X$  is  smooth and geometrically  irreducible  
of dimension $(n-l-c)(l+c)$.
\end{proposition}
\begin{example}\label{example:Nn11first}
Let $(x_1:\dots: x_n)$ and $(y_1:\dots: y_n)$
be  homogeneous coordinates of $\Gr_{V,1}=\P_*(V)$ and $\Gr_V^1=\P^*(V)$,
respectively, 
that are dual to each other.
Then the incidence subvariety $\Incvar$
 is  defined by  $\sum x_iy_i=0$ in $\P_*(V)\times \P^*(V)$, and hence  
$X[r,s]^1_1$ is  defined  by 
\begin{equation}\label{eq:Mukai}
\renewcommand{\arraystretch}{1.2}
\left\{
\begin{array}{l}
 x_1^r\, y_1+\cdots +x_n^r\, y_n=0, \\
 x_1\, y_1^{s}+\cdots +x_n\, y_n^{s}=0.
\end{array}
\right.
\end{equation}
%
%
Therefore
$X[r, s]_{1}^1$ is of general type when $r$ and $s$ are sufficiently large.
\end{example}
We show that the Frobenius incidence varieties,
which are of  non-negative  Kodaira dimension in general,
have two pathological features of 
algebraic geometry in positive characteristics;
namely, supersingularity and unirationality.
\par
\medskip
Our first main result is as follows:
\begin{theorem}\label{thm:Fss}
There exists a polynomial $N(t)$ with integer coefficients
such that the number 
of $\F_{(rs)^\nu}$-rational points of $X$
is equal to $N((rs)^\nu)$ for any $\nu\in \Z_{>0}$.
\end{theorem}
In other words,
$X$ is supersingular over $\F_{rs}$
in the sense that 
the eigenvalues of the $rs$\,th power Frobenius 
endomorphism acting  on the $l$-adic  cohomology ring of $X\tensor \bar{\F}_{rs}$
 are powers of $rs$ by integers.
\par
\medskip
We also give in Theorem~\ref{thm:Nnlc}
a recursive formula for 
the polynomial $N(t)$.
We see that
the odd Betti numbers of $X$ are zero, and 
can calculate the even Betti numbers $b_{2i}(X)$ of $X$
via the formula
\begin{equation}\label{eq:NBetti}
N(t)={\sum_{i=0}^{\dim X}} b_{2i}(X)\,t^i.
\end{equation}
\begin{example}\label{example:Nn11second} 
The Betti numbers of $X[r, s]_{1}^1$ in Example~\ref{example:Nn11first} are 
\begin{equation*}\label{eq:bettiXn11}
b_{2i}=b_{2(n-2)-2i}=
\begin{cases}
i+1 & \textrm{if $i<n-2$,}\\
n-2+((rs)^{n}-1)/(rs-1) & \textrm{if $i=n-2$.}
\end{cases}
\end{equation*}
\end{example}
%
%
%
%
%
The number of rational points of the Deligne-Lusztig varieties has been studied 
by means of the representation theory
of algebraic groups over finite fields. 
See, for example, ~\cite{MR723216} and~\cite{MR747534}.
Our proof of Theorems~\ref{thm:Fss} and~\ref{thm:Nnlc}  does not use the representation theory,
and is entirely elementary.
\par
\medskip
%
Our next result is on the unirationality of the Frobenius incidence varieties.
A variety $Y$ defined over $\F_q$ is said to be 
\emph{purely-inseparably unirational over $\F_q$}
if there exists a purely inseparable dominant rational map $\P^N\ratmap Y$
defined over $\F_q$. 
\begin{theorem}\label{thm:unirational}
The Frobenius incidence variety $X$ is purely-inseparably unirational over $\F_{p}$.
\end{theorem}
The relation of  supersingularity to unirationality has been observed in various cases.
 See Shioda~\cite{MR0374149} for the supersingularity of unirational surfaces, and 
 see  Shioda-Katsura~\cite{MR526513} and Shimada~\cite{MR1176080}
 for the unirationality of supersingular Fermat varieties.
\par
\medskip
From the defining equations~\eqref{eq:Mukai} of $X[r, s]_{1}^1$,
we see that $X[r, s]_{1}^1$ is a complete intersection of 
two varieties of unseparated flags~\cite{MR1247497},
or more specifically, of two unseparated incidence varieties~\cite[\S2]{MR1385284}.
Varieties of unseparated flags are classified in~\cite{MR1096262} and~\cite{MR1247497}.
Their pathological property  with respect to Kodaira vanishing theorem 
was studied in~\cite{MR1183809}.
\par
\medskip
Next we investigate algebraic cycles on the Frobenius incidence varieties. 
%
%
Let $\Lambda$ be an $\F_{rs}$-rational linear subspace of $V_{\F}:=V\tensor \F_{rs}$
such that $l\le \dim\Lambda\le n-c$.
We define a   subvariety
$\Sigma_\Lambda$ of $ \Gr_{n, l}\times \Gr_{n}^c$ 
by 
\begin{equation}\label{eq:defSigma}
\Sigma_\Lambda:=\Gr_{\Lambda, l}\times \Gr_{V_{\F}/\Lambda^{(r)}}^c.
\end{equation}
Then $\Sigma_\Lambda$ is defined over $\F_{rs}$ and,
for any field $F$ over  $\F_{rs}$,
we have 
$$
\Sigma_\Lambda(F)=
\set{(L, M)\in \Gr_{n, l}(F)\times \Gr_{n}^c (F)}%
{L\subset\Lambda\;\;\rmand\;\; \Lambda\sp{r}\subset M}.
$$
It follows from $\Lambda\sp{rs}=\Lambda$
that $\Sigma_\Lambda$ is contained in $X$.
In Theorem~\ref{thm:intnumb},
we calculate the intersection  of these 
 algebraic cycles $\Sigma_\Lambda$ in the Chow ring $A(X)$ of $X$.
 (See~\cite[App.~A]{MR0463157} or~\cite{MR1644323} for the definition of Chow rings.)
\par
\medskip
Applying  Theorem~\ref{thm:intnumb} to the case $l=c=1$,
we investigate the lattice generated by the numerical equivalence classes of 
middle dimensional algebraic cycles of 
$$
X^1_1=X[r,s]^1_1\subset \Gr_{V,1}\times \Gr_V^1=\P_*(V)\times \P^*(V).
$$
Note that, when  $l=c$, we have 
$2\dim \Sigma_\Lambda=\dim X[r, s]^l_l$ for any $\Lambda$.
%
%
%
Let 
 $A^{n-2} (X^1_1)$ denote the  Chow group
of  middle-dimensional algebraic cycles on $X^1_1$ over $\bar{\F}_{p}$.
For $i=1, \dots, n-1$,
let  $h_i$ be   the intersection of $X^1_1$
with $P_i\times P_{n-i}\subset \P_*(V)\times \P^*(V)$,
where $P_j$ is a general $j$-dimensional projective linear subspace
of $\P_*(V)$ or $\P^*(V)$.
Then $h_i$ is of middle-dimension on $X^1_1$.
We consider  the submodule 
$$
\latNtl (X^1_1)\;\subset\; A^{n-2} (X^1_1)
$$
generated  
by   $h_1, \dots, h_{n-1}$ 
and $\Sigma_{\Lambda}$ associated with 
all $\F_{rs}$-rational linear subspaces $\Lambda$ of  $V\tensor {\F_{rs}}$
such that   $1\le\dim\Lambda\le n-1$.
Then we have the intersection pairing on $\latNtl (X^1_1)$.
Let $\latNtl (X^1_1)\sperp$ denote  the  submodule of $\latNtl (X^1_1)$ consisting of  
$x\in \latNtl (X^1_1)$ such that $(x, y)=0$ holds for any  $y\in \latNtl (X^1_1)$.
We set
$$
\latN(X^1_1):=\latNtl (X^1_1)/\latNtl (X^1_1)\sperp.
$$
Then $\latN(X^1_1)$ is a finitely-generated free $\Z$-module equipped with 
the  non-degenerate intersection pairing
\begin{equation*}\label{eq:Npairing}
\latN(X^1_1)\times \latN(X^1_1) \to \Z.
\end{equation*}
Thus  $\latN(X^1_1)$ is a lattice.
In  Theorem~\ref{thm:NNN}, we describe  the rank and the discriminant  of this lattice.
As  a corollary of Theorem~\ref{thm:NNN}, we obtain the following:
\begin{corollary}\label{cor:cohring}
The $l$-adic cohomology ring of $X^1_1\tensor \bar{\F}_{rs}$ is generated by 
the cohomology classes of the algebraic cycles $\Sigma_{\Lambda}$ and the image 
of the restriction homomorphism from the cohomology ring of $\P_*(V)\times \P^*(V)$.
\end{corollary}
%
%
%
%
In Theorem~\ref{thm:NNN}, it is shown that 
the discriminant of $\latN(X^1_1)$ is a power of $p$.
This fact is 
an analogue of the theorem  on
the discriminant of 
the  N\'eron-Severi lattice of a supersingular $K3$ surface
(in the sense of Shioda)
proved by  
Artin~\cite{MR0371899} and  Rudakov-Shafarevich~\cite{MR633161}.
See also~\cite{MR1794260}
for a similar result on the Fermat variety of degree $q+1$.
%
%
%
\par
\medskip
For $x\in \latNtl(X^1_1)$, let $[x]\in \latN(X^1_1)$ denote the class of $x$
modulo $\latNtl(X^1_1)\sperp$.
We define the primitive part $\Nprim(X^1_1)$ of  $\latN(X^1_1)$ by 
$$
\Nprim(X^1_1):=\set{[x]\in \latN(X^1_1)}{([x], [h_i])=0\;\;\textrm{for}\;\; i=1, \dots, n-1}.
$$
For a lattice $L$, let $[-1]^\nu L$ denote the lattice obtained from $L$
by multiplying the symmetric bilinear form 
with $(-1)^\nu$.

\begin{theorem}\label{thm:NNNprim}
The intersection pairing on $\Nprim(X^1_1)$ is non-degenerate.
The lattice $[-1]^n\Nprim(X^1_1)$ is  positive-definite
of rank 
$|\P^{n-1}(\F_{rs})|-1$.
\end{theorem}
In the last section,  
our construction is applied to the sphere packing problem.
We investigate the case $n=4$, and study
the positive-definite lattice $\Nprim (X[2,2]^1_1)$ of 
the $4$-dimensional  Frobenius incidence variety $X[2,2]^1_1$.
%
%
%
%
%
\begin{theorem}\label{thm:Nprim}
Suppose that $n=4$.
The lattice  $\Nprim(X[2,2]^1_1)$ is an even positive-definite lattice of rank $84$,
 with discriminant $85\cdot 2^{16}$,
 and with minimal  norm  $8$.
\end{theorem}
In particular, the normalized center density of $\Nprim(X[2,2]^1_1)$  is 
$2^{34}/\sqrt{85}=2^{30.795...}$,
while the Minkowski-Hlawka bound at
rank $84$ 
is $2^{17.546...}$.
See \S\ref{sec:dense}
for the definition of normalized center density and Minkowski-Hlawka bound.

\par
\medskip
In the proof of Theorem~\ref{thm:Nprim},
we  construct another  positive-definite lattice $\MMM_{\CCC}$ of rank $85$
associated with a code $\CCC$ over $\Z/8\Z$.
The normalized  center  density $2^{32.5}$ of $\MMM_{\CCC}$ is also larger than 
the  Minkowski-Hlawka bound $2^{18.429\dots}$ at rank $85$.
See~Theorem~\ref{thm:MCCC}.
\subsection{The plan of this paper}
The  proofs of these results are  given as follows. 
In~\S\ref{sec:betti}, we show that the Frobenius incidence variety $X$ is smooth 
in Proposition~\ref{prop:FFF},
and give a recursive formula for the number of
$\F_{(rs)^\nu}$-rational points of $X$ in Theorem~\ref{thm:Nnlc}.
Proposition~\ref{prop:smoothdim} and 
Theorem~\ref{thm:Fss} follow from these results.
In~\S\ref{sec:unirational}, 
we show that $X$ is purely-inseparably  unirational.
In~\S\ref{sec:intnumb},  
we give a formula for 
the intersection  of the algebraic cycles 
$\Sigma_{\Lambda}$ in the Chow ring of  $X$. 
In~\S\ref{sec:latticeN},
we study the case  where  $l=c=1$,
and prove Corollary~\ref{cor:cohring} and Theorem~\ref{thm:NNNprim}.
In the last section,
we study the case $n=4$, $l=c=1$, $r=s=2$,  and prove Theorem~\ref{thm:Nprim}.
\section{Number of rational points and  the Betti numbers}\label{sec:betti}
In this section,
we prove Proposition~\ref{prop:smoothdim} and Theorem~\ref{thm:Fss}.
\par
\medskip
It is useful to note that the defining property~\eqref{eq:setX} of 
the Frobenius incidence variety $X=X[r,s]_l^c$
is rephrased as follows:
\begin{equation}\label{eq:setX2}
\renewcommand{\arraystretch}{1.2}
\begin{array}{ccl}
X(F)&=&
\set{(L, M)\in {\Gr}_{n,l}(F)\times {{{\Gr}_n^c}}(F)}%
{L\sp{r}\subset M \;\;\rmand\;\; L\subset M\sp{s}}\\
&=&
\set{(L, M)\in {\Gr}_{n,l}(F)\times {{{\Gr}_n^c}}(F)}%
{L+L\sp{rs}\subset M\sp{s}}\\
&=&
\set{(L, M)\in {\Gr}_{n,l}(F)\times {{{\Gr}_n^c}}(F)}%
{L\sp{r}\subset M\cap M\sp{rs}}.
\end{array}
\end{equation}
\par
\medskip
We denote by 
$$
S_{n,l}\to{\Gr}_{n,l}\quand Q_n^c\to{{\Gr}_n^c}
$$
 the universal  subbundle of $V\tensor\OOO\to {\Gr}_{n,l}$
and  the universal quotient bundle of $V\tensor\OOO\to {{\Gr}_n^c}$,
respectively.
We consider the  vector bundle
$$
\EEE:=\sheafHom(\pr^* ({S_{n,l}}), \pr^*({Q_{n}^c}))\to\GGlc
$$
of rank $lc$,
where $\pr$ denotes the projections $\GGlc\to{\Gr}_{n,l}$ and $\GGlc\to{{\Gr}_n^c}$.
Let $\gamma: \GGlc\to \EEE$ denote the section of $\EEE$
corresponding to the canonical homomorphism
$$
\pr^* ({S_{n,l}})\inj V\tensor\OOO_{\GGlc} \surj \pr^*({Q_{n}^c}).
$$ 
We then put
\begin{equation}\label{eq:defFFF}
\FFF:=(\Fr\spar{r} \times \id )^*\EEE\oplus ( \id \times  \Fr\spar{s})^*\EEE,
\end{equation}
which is a vector bundle over $\GGlc$ of rank $2lc$ that has  a canonical section
$$
\tilde\gamma:=(\,(\Fr\spar{r} \times \id )^*\gamma, \, ( \id \times  \Fr\spar{s})^*\gamma\,)\;\;:\;\; \GGlc\to \FFF.
$$
Since the incidence variety $\Incvar$ is defined on $\GGlc$ by $\gamma=0$,
the subscheme $X$ of $\GGlc$ is defined by $\tilde\gamma=0$.
\begin{proposition}\label{prop:FFF}
The section $\tilde\gamma$ intersects the zero section of $\FFF$ transversely
in the total space of $\FFF$.
In particular,
the scheme $X$ is smooth of dimension 
$$
\dim(\GGlc)-2lc=(l+c)(n-l-c).
$$
\end{proposition}
\begin{proof}
It is enough to show that,
for any field $F$ of characteristic $p$,
the  tangent space to $X$ at an arbitrary $F$-valued point 
of $X$ is of dimension $(l+c)(n-l-c)$.

Let $(L,M)$ be an $F$-valued point of $\Incvar$.
Then the tangent space to 
$\GGlc$ at $(L, M)$ is canonically identified with
the linear space
$$
T(L, M):=\Hom(L, V_F/L)\oplus \Hom(M, V_F/M),
$$
and  the tangent space to 
$\Incvar$ at $(L, M)$ is identified with
the linear subspace of
$T(L, M)$ 
consisting of pairs $(\alpha, \beta)\in T(L, M)$
that make the following diagram commutative:
\begin{equation}\label{eq:squarealphabeta}
\renewcommand{\arraystretch}{1.2}
\begin{array}{ccc}
L&\inj  & M\\
\lower2pt \hbox{${}^\alpha$}\downarrow\phantom{\hbox{${}^\alpha$}} && 
\phantom{\hbox{${}^\beta$}}\downarrow\lower2pt \hbox{${}^\beta$} \\
 {V_F}/L&\surj &{V_F}/M,
\end{array}
\end{equation}
where the horizontal arrows are the natural linear maps.

We now let  $(L,M)$ be an $F$-valued point of $X$.
Note that 
the Frobenius morphism induces the  zero map on the tangent space.
Suppose that $r>1$ and $s>1$.
Then 
the tangent space to $X$ at $(L, M)$ is  identified with
the linear subspace of $T(L, M)$ consisting of  pairs $(\alpha, \beta)$
that make
both  triangles 
\begin{equation*}\label{eq:triangle}
\renewcommand{\arraystretch}{1.2}
(\textrm{T}_\beta)\;\;\;\;\;\;\;
\begin{array}{ccc}
L\sp{r}&\inj  & M\\
\lower4pt \hbox{${}^0$}\searrow \hskip -25pt && 
\phantom{\lower4pt \hbox{${}^\beta$}}\downarrow{\lower4pt \hbox{${}^\beta$}}\\
& &{V_F}/M
\end{array}
\quad\quad\quand\quad\quad
(\textrm{T}_\alpha)\;\;\;\;\;\;\;
\renewcommand{\arraystretch}{1.2}\begin{array}{ccc}
L& & \\
\lower2pt \hbox{${}^\alpha$} \downarrow\phantom{\lower2pt \hbox{${}^\alpha$}}
 & \hskip 0pt \searrow  \hskip 5pt  \lower2pt \hbox{${}^0$} \hskip -25pt&\\
{V_F}/L&\surj &{V_F}/M\sp{s}
\end{array}
\end{equation*}
commutative.
Suppose that $r=1$ and $s>1$
(resp.~$r>1$ and $s=1$).
Then  the tangent space to $X$ at $(L, M)$ is  identified with
the linear subspace  of  pairs $(\alpha, \beta)$
that make both of~\eqref{eq:squarealphabeta} and $(\textrm{T}_\alpha)$ 
(resp.~\eqref{eq:squarealphabeta} and~$(\textrm{T}_\beta)$) commutative.
In each case, one easily checks that 
the dimension of 
the tangent space  is  $(l+c)(n-l-c)$.
\end{proof}
Next we count the number of $\F_{(rs)^\nu}$-rational points of $X$.
In order to state the result,
we need to introduce several polynomials.
\par
\medskip
For each integer $l$ with $0\le l\le n$, 
we define 
a polynomial $g_{n, l} (x)=g_n^{n-l}(x)\in \Z[x]$ of degree $l(n-l)$ by
$$
g_{n, l} (x)=g_n^{n-l}(x):=\frac{\prod_{i=0}^{l-1} (x^n-x^i)}{\prod_{i=0}^{l-1} (x^l-x^i)}. 
$$
Note that $g_{n, l} (x)$ is monic if $l(n-l)>0$,
while $g_{n, 0} (x)=g_{n,n}(x)=1$.
We also put
$$
g_{n, l} (x)=g_n^{n-l}(x):=0\;\;\textrm{for \,$l<0$\, or\,  $l>n$.} 
$$
Then 
the number of $\F_{q^\nu}$-rational points of 
${\Gr}_{n,l}=\Gr_n^{n-l}$ is equal to $g_{n, l} (q^\nu)=g_n^{n-l} (q^\nu)$.
Let  $\succ $  denote the lexicographic order on the set of
pairs $(l, d)$ of  non-negative integers $l$ and $d$:
\begin{equation*}\label{eq:ord}
(l, d)\succ  (l\sprime, d\sprime)\;\;\Longleftrightarrow\;\;
l>l\sprime \;\textrm{or}\; (l=l\sprime \sqand d>d\sprime).
\end{equation*}
By descending induction with respect to 
$\succ$, 
we define  polynomials $\tau_{l,d}(x, y) \in \Z[x,y]$ 
as follows:
\begin{equation}\label{eq:tauini}
\tau_{ l,  d} (x,y):=
\begin{cases}
0 & \textrm{if $l>n$ or $d>l$,} \\
g_{n, l}(x) & \textrm{if $d=l\le n$,} 
\end{cases}
\end{equation}
and, for $d<l\le n$, by
\begin{equation}\label{eq:taurec}
\tau_{l,  d} (x,y):=
\sum_{u=l}^{2l-d} \tau_{2l-d, u} (x, y) \cdot g_{u, l}(y)-
\sum_{t=d+1}^l \tau_{ l,  t} (x, y) \cdot g_{n-2l+t, t-d}(y).
\end{equation}
Finally, for positive integers $l$ and $c$ with $l+c<n$,  we put
\begin{equation}\label{eq:NX}
\Nlc(x, y):=\sum_{d=0}^l \tau_{ l,  d} (x, y)\cdot g_{n-2l+d}^c(y)\; \in\; \Z[x,y].
\end{equation}
The main result of this section is as follows:
\begin{theorem}\label{thm:Nnlc}
The polynomial $\Nlc(x, y)$ is monic of degree $(l+c)(n-l-c)$ with respect to 
the variable $y$,  
and
the number of $\F_{(rs)^\nu}$-rational points of 
$X=X[r,s]^c_l$ is equal to $\Nlc(rs, (rs)^\nu)$.
\end{theorem}
Theorem~\ref{thm:Nnlc} provides us with an algorithm to calculate the Betti numbers of $X$
by~\eqref{eq:NBetti}.
From Proposition~\ref{prop:FFF} and 
the fact that $\Nlc (x, y)$ is monic with respect to $y$,
we obtain the following:
\begin{corollary}
The Frobenius incidence variety $X$ is geometrically  irreducible.
\end{corollary}
Thus the proof of Proposition~\ref{prop:smoothdim} and 
Theorem~\ref{thm:Fss} will be  completed by Theorem~\ref{thm:Nnlc}.
\par
\medskip
For the proof of Theorem~\ref{thm:Nnlc},
we let $q$ be a power of $p$ by a positive integer,
and  define  locally-closed reduced subvarieties
$\Tq_{ l,  d}$ of $\Gr_{n, l}$  over $\F_p$ by the property that 
\begin{equation}\label{eq:Tq}
\Tq_{ l,  d}(F)=\set{L\in {\Gr}_{n,l}(F)}{\dim (L\cap L^q )=d}
\end{equation}
should hold for any field $F$ of characteristic $p$.
First we prove the following:
\begin{proposition}\label{prop:Ttau}
For  any pair $(l, d)$ of non-negative integers $l$ and $d$, 
the number 
of $\F_{q^\nu}$-rational points of $\Tq_{ l,  d}$ is equal to 
$\tau_{ l,  d} (q, q^\nu)$.
\end{proposition}
\begin{proof}
We proceed by descending induction on $(l, d)$ with respect to $\succ$.
By definition, we have $\Tq_{ l,  d} (\F_{q^\nu})=\emptyset$ for 
any $\nu\in \Z_{>0}$
if  $l>n$ or $d>l$.
Since $L=L^q$ is equivalent to the condition that $L$ be $\F_q$-rational, 
we have
$$
\Tq_{ l,  l} (\F_{q^\nu})=
\Gr_{n, l} (\F_{q})\;\; \textrm{for all\, $\nu\in \Z_{>0}$}.
$$
Thus  $|\Tq_{ l,  d} (\F_{q^\nu})|=\tau_{ l,  d} (q, q^\nu)$ 
holds for any $(l, d)$ with $l>n$ or $d\ge l$ by~\eqref{eq:tauini}.
Suppose that $d<l\le n$
and that 
$|\Tq_{ l\sprime,  d\sprime} (\F_{q^\nu})|=\tau_{ l\sprime, d\sprime}(q, q^\nu)$
holds for any $(l\sprime, d\sprime)$
with $(l\sprime, d\sprime)\succ (l, d)$.
We count 
the number of the elements of the finite set
$$
\PPP_{l,d}:=\set{(L, M)\in \Gr_{n, l} (\F_{q^\nu}) \times \Gr_{n, 2l-d} (\F_{q^\nu}) }%
{L+L^q \subset M}
$$
in two ways.
If  $(L, M)\in \PPP_{l,d}$,
then we have $d \le \dim (L\cap L^q )\le l$.
If $L\in \Tq_{ l, t} (\F_{q^\nu})$ with $d\le t\le l$,
then $\dim (L+ L^q )= 2l-t$ holds and the number of 
$M\in \Gr_{n, 2l-d} (\F_{q^\nu})$ containing  $L+L^q$
is equal to $g_{n-2l+t, t-d}(q^\nu)$.
Hence we have
\begin{equation}\label{eq:LHS}
|\PPP_{l,d}|=\sum_{t=d}^l |\Tq_{ l,  t} (\F_{q^\nu})| \cdot g_{n-2l+t, t-d}(q^\nu).
\end{equation}
On the other hand,
a pair $(L, M)\in \Gr_{n, l} (\F_{q^\nu}) \times \Gr_{n, 2l-d} (\F_{q^\nu})$ 
satisfies $L+L^q \subset M$
if and only if $L^q \subset M\cap M^q $, or equivalently,  
if and only if $L \subset M\cap M\sp{1/q}$.
Note that $M\sp{1/q}$ is also $\F_{q^\nu}$-rational.
If $L^q \subset M\cap M^q $ holds, then
we have $l\le \dim (M\cap M^q )\le 2l-d$.
If $\dim (M\cap M^q )=u$   
with $l\le u\le 2l-d$,
then the number of $L\in \Gr_{n, l} (\F_{q^\nu})$ contained in  $M\cap M\sp{1/q}$
is equal to $g_{u, l}(q^\nu)$.
Hence we have
\begin{equation}\label{eq:RHS}
|\PPP_{l,d}|=\sum_{u=l}^{2l-d} |\Tq_{2l-d,u}(\F_{q^\nu})| \cdot g_{u, l}(q^\nu).
\end{equation}
Comparing~\eqref{eq:LHS} and~\eqref{eq:RHS},
we obtain 
\newcommand{\tsum}{\textstyle\sum}
\newcommand{\dsum}{\displaystyle\sum}
\begin{eqnarray*}
&& |\Tq_{ l,  d} (\F_{q^\nu})|\cdot g_{n-2l+d, 0}(q^\nu)\mystruthd{0pt}{10pt}\\
 &=&
{\displaystyle\sum_{u=l}^{2l-d}} |\Tq_{ 2l-d, u} (\F_{q^\nu})| \cdot  g_{u, l}(q^\nu)-
\sum_{t=d+1}^l|\Tq_{ l,  t} (\F_{q^\nu})| \cdot g_{n-2l+t, t-d}(q^\nu)\\
&=& \tau_{ l,  d} (q, q^\nu)
\end{eqnarray*}
by the induction hypothesis and the defining equality~\eqref{eq:taurec}.
If $n-2l+d<0$, then 
$g_{n-2l+d, 0}(x)=0$ and  hence we have $\tau_{ l,  d} (q, q^\nu)=0$,
while we have
$\Tq_{ l,  d} (\F_{q^\nu})=\emptyset$ because $L\in \Tq_{ l,  d} (\F_{q^\nu})$
would imply $\dim (L+L^q)>n$.
If $n-2l+d\ge 0$, then 
we have $|\Tq_{ l,  d} (\F_{q^\nu})|=\tau_{ l,  d} (q, q^\nu)$
because $g_{n-2l+d, 0}(q^\nu)=1$.
Therefore $|\Tq_{ l,  d} (\F_{q^\nu})|=\tau_{ l,  d} (q, q^\nu)$
is proved for any  $(l, d)$.
\end{proof}
Next we put
$$
\delta(l, d):=(l-d)(n-l+d),
$$
and prove the following:
\begin{proposition}\label{prop:tau}
Consider the following condition:
$$
\cond (l, d): \;\; \max(0, 2l-n)\le d\le l\le n.
$$
If $\cond(l, d)$ is false, then $\tau_{ l,  d}(x, y)=0$.
If $\cond(l, d)$ is true, then
$\tau_{ l,  d}(x, y)$ is non-zero and of degree $\delta(l, d)$ with respect to $y$.
If $\cond(l, d)$ is true and $\delta(l, d)>0$,
then $\tau_{ l,  d}(x, y)$ is monic with respect to $y$.
\end{proposition}
\begin{proof}
First remark that,
if $a(x, y)\in \Z[x, y]$ satisfies $a(q, q^\nu)=0$ 
for any prime powers $q$ and any positive integers $\nu$,
then we have $a(x, y)=0$.
\par
If $\cond(l, d)$ is false, then
$\Tq_{ l,  d}$ is an empty variety   
for any $q$ by definition,
and hence $\tau_{ l,  d}(x, y)=0$ by Proposition~\ref{prop:Ttau}.
\par
We prove  the assertion
$$
\statement (l, d)
\quad
\left\{
\renewcommand{\arraystretch}{1.2}
\begin{array}{l}
\cond (l, d)\;\;\Rightarrow\;\; \tau_{ l,  d} (x, y)\ne 0 \;\;\textrm{and}\; \deg_y \tau_{ l,  d} (x, y)=\delta(l, d),\\
\cond (l, d) \sqand \delta(l,d)>0\;\;\Rightarrow\;\; \textrm{$\tau_{ l,  d}(x, y)$ 
is monic with respect to $y$},
\end{array}
\right.
$$
by descending induction on $(l, d)$ with respect to 
the order $\succ$.
If $\cond (l, d)$ is false, then  $\statement (l, d)$ holds vacuously.
Hence we can assume 
that $\statement (l\sprime, d\sprime)$ holds for any 
$(l\sprime, d\sprime)$ with $(l\sprime, d\sprime)\succ (l, d)$,
and  that $\cond (l, d)$ is true.
If $d=l$, then $\statement (l, d)$  holds 
because $\tau_{ l,  l}(x, y)=g_{n, l}(x)$ is a non-zero constant with respect to $y$.
We assume that $d<l$. 
Note that now we have $n-l+d\ge l>d$ and 
$\delta(l, d)>0$.
First we study the second summation of~\eqref{eq:taurec}.
For $t$ with $d+1\le t\le l$, we have  
$$
\delta(l, t)+\deg g_{n-2l+t, t-d}(y)=\delta(l, d)-t(t-d)<\delta(l, d).
$$
Hence, by the induction hypothesis,
every  term in the second summation  is non-zero of degree
$<\delta(l, d)$ with respect to $y$.
Next we study the first summation of~\eqref{eq:taurec}.
Note that the condition $\cond(2l-d, u)$ on $u$ is 
$$
\max(0, 4l-2d-n)\le u\le 2l-d.
$$
From the induction hypothesis,
the non-zero terms in  the first summation
are 
$$
s_u:=\tau_{ 2l-d, u} (x, y) \cdot g_{u, l}(y)
\;\;\;\textrm{with}\;\;\max(l, 4l-2d-n)\le u\le 2l-d.
$$
By the equality
$$
\delta (2l-d, u)+\deg  g_{u, l}(y)=\delta(l, d)-(u-l)(u-4l+2d+n),
$$
every non-zero term in  the first summation is of degree $\le \delta(l, d)$
with respect to $y$.
Moreover there exists one and only one term 
of degree equal to $ \delta(l, d)$, which is
\begin{equation}\label{eq:maxdegFS}
\begin{cases}
s_l=\tau_{ 2l-d, l} (x, y) \cdot g_{l, l}(y)&\textrm{if $l\ge 4l-2d-n$,} \\
s_{4l-2d-n}=\tau_{ 2l-d,4l-2d-n} (x, y) \cdot g_{4l-2d-n, l}(y) 
&\textrm{if $l<4l-2d-n$}. \\
\end{cases}
\end{equation}
It remains to show that this term is monic with respect to $y$.
In the case where $l\ge 4l-2d-n$, 
the term
$s_l$ 
is monic with respect to $y$, 
because $g_{l,l}(y)=1$ and 
$\delta(2l-d, l)=\delta(l, d)>0$.
We consider  the case  $l<4l-2d-n$.
Note that
$$
\delta_1:=\delta(2l-d,4l-2d-n)=(n-2l+d)(2l-d)\ge 0,
$$
and that $\delta_1=0$  holds if and only if $n-2l+d=0$,
because we have $2l-d>l\ge 0$.
Suppose that $\delta_1>0$.
Then 
$\tau_{ 2l-d,4l-2d-n} (x, y)$
is monic of degree $\delta_1$ with respect to $y$
by the induction hypothesis.
If $\delta_1=0$,
then we have
$$
\tau_{ 2l-d,4l-2d-n} (x, y)=\tau_{ n,n} (x, y)=g_{n,n}(x)=1.
$$
On the other hand,  $g_{4l-2d-n, l}(y)$ is monic of degree $>0$
for $l<4l-2d-n$.
Therefore  the term $s_{4l-2d-n}$ in the second case of~\eqref{eq:maxdegFS} is 
also monic with respect to $y$.
Thus the statement $\statement(l, d)$ holds.
\end{proof}
\begin{proof}[Proof of Theorem~\ref{thm:Nnlc}]
We show that
 $\Nlc(x, y)$ is monic of degree
$$
\tilde\delta(l,c):=(l+c)(n-l-c)
$$
with respect to $y$. 
We put
$$
t_d:=\tau_{ l,  d} (x, y)\cdot g_{n-2l+d}^c(y),
$$
so that $\Nlc(x, y)=\sum_{d=0}^l t_d$.
If $d<2l-n+c$, then $t_d=0$ because $g_{n-2l+d}^c(y)=0$.
Suppose that $d\ge 0$ satisfies 
$d\ge 2l-n+c$.
Then  we have
$$
\delta(l, d)+\deg g_{n-2l+d}^c (y)=\tilde\delta(l, c)-d(d-2l+n-c)\le \tilde\delta(l, c).
$$
Hence every non-zero term $t_d$ is of degree $\le \tilde\delta(l, c)$
with respect to $y$, and 
there are at most two terms 
that are of degree equal to $\tilde\delta(l, c)$;
namely, 
$$
t_0=\tau_{ l,  0}(x, y)\cdot g_{n-2l}^c (y)
\qquand 
t_{2l-n+c}=\tau_{ l,  2l-n+c}(x, y)\cdot g_{c}^c(y).
$$
If  $2l-n+c<0$,
then $t_{2l-n+c}$ does not appear in the summation $\sum_{d=0}^l t_d$,
and $t_{0}$ is non-zero and monic with respect to $y$
because $\cond(l, 0)$ holds,
$\delta(l, 0)>0$ and  $\deg g_{n-2l}^c (y)>0$.
If  $2l-n+c=0$,
then $t_{2l-n+c}=t_{0}$  is non-zero and monic  with respect to $y$
because $\cond(l, 0)$ holds, $\delta(l, 0)>0$ and $g_{n-2l}^c (y)=g_{c}^c (y)=1$.
If  $2l-n+c>0$,
then $t_{0}=0$  because $g_{n-2l}^c (y)=0$,
and the term $t_{2l-n+c}=\tau_{ l,  2l-n+c}(x, y)$,
which appears in the summation since $2l-n+c<l$, 
is non-zero and monic with respect to $y$
because $\cond(l, 2l-n+c)$ holds and $\delta(l, 2l-n+c)>0$.
Thus, in each case,
there exists one and only one  term 
that is non-zero of degree  $\tilde\delta(l, c)$,
and this term is monic with respect to $y$.
Hence the assertion is proved.
\par
\medskip
Finally we prove that the number of $\F_{(rs)^\nu}$-rational points of 
$X=X[r,s]^c_l$ is 
equal to $\Nlc(rs, (rs)^\nu)$.
For simplicity, we put $q:=rs$.
By the property~\eqref{eq:setX2} of $X$,  it is enough to show that
 the number of
the pairs 
$(L, M)\in \Gr_{n,l}(\F_{q^\nu})\times \Gr_{n}^c(\F_{q^\nu})$
satisfying $L+L\sp{q}\subset M\sp{s}$
is 
equal to $\Nlc(q, q^\nu)$.
Note that $\Gr_{n,l}(\F_{q^\nu})$ is the disjoint union of the finite sets 
$\Tq_{l, d}(\F_{q^\nu})$
over $d$ with $0\le d\le l$.
If $L\in \Gr_{n,l}(\F_{q^\nu})$ is contained in $\Tq_{l, d}(\F_{q^\nu})$,
then
$L+L\sp{q}$ is of dimension $2l-d$ and hence the number of 
$M\sprime\in \Gr_{n}^c(\F_{q^\nu})$
containing $L+L\sp{q}$
is $g_{n-2l+d}^c(q^\nu)$.
Because $M\mapsto M\sp{s}$ is a bijection from $\Gr_{n}^c(\F_{q^\nu})$ to itself,
the number of the pairs  is equal to
the sum of $\tau_{ l, d}(q, q^\nu)\cdot g_{n-2l+d}^c(q^\nu)$ 
over $d$ with $0\le d\le l$ by Proposition~\ref{prop:tau}. 
Thus we have $|X(\F_{q^\nu})|=\Nlc(q, q^\nu)$
by the definition~\eqref{eq:NX}. 
\end{proof}
The following 
is useful in the computation of $\Nlc(x, y)$:
\begin{corollary}\label{cor:dual}
We have
$\tau_{l,d}(x, y)=\tau_{ n-l, n-2l+d}(x, y)$
for any $(l, d)$.
\end{corollary}
\begin{proof}
We choose an inner product
$V\times V\to \F_p$
defined over $\F_p$,
and denote by $L\sperp\subset V_F$ the orthogonal complement of a linear subspace 
$L\subset V_F$.
Then $L\mapsto L\sperp$ induces an isomorphism $\Gr_{V, l}\isom \Gr_{V, n-l}$.
Since $(L^q)\sperp=(L\sperp)^q$
and $L\sperp\cap (L\sperp)^q=(L+L^q)\sperp$,
this isomorphism induces a bijection from  
$\Tq_{ l, d}(\F_{q^\nu})$ to $ \Tq_{ n-l, n-2l+d}(\F_{q^\nu})$.
Thus the  equality  follows from
Proposition~\ref{prop:Ttau} and the remark at the beginning of the proof of Proposition~\ref{prop:tau}.
\end{proof}
\begin{example}\label{example:Nn11third}
We have
$$
N_{1}^1(x, y)=g_{n, 1}(y) \cdot g_{n-2}^1(y) +g_{n, 1}(x) \cdot (g_{n-1}^1(y) -g_{n-2}^1(y) ), 
$$
and hence the Betti numbers of $X[r, s]_{1}^1$ in Example~\ref{example:Nn11second} are obtained.
\end{example}
\begin{example}\label{example:Nn22}
We have
\begin{eqnarray*}
\tau_{ 2, 1}(x, y) &=&  \tau_{ n-2, n-3}(x, y)\\
&=& g_{n, n-1}(y)-g_{n, n-1}(x) + g_{n, n-1}(x)\cdot  g_{n-1, n-2}(y) - 
g_{n, n-2}(x)\cdot  g_{2, 1}(y),
\end{eqnarray*}
and hence
\newcommand{\gx}[2]{g_{#1}^{#2} (x)}
\newcommand{\gy}[2]{g_{#1}^{#2} (y)}
\begin{eqnarray*}
N_{2}^2(x, y)&=&\phantom{\,+\,}
\gy{n}{ 2}\cdot \gy{n-4}{ 2}\,+\, \gy{n}{ 1}\cdot \gy{n-3}{ 2}\,+\, \gx{n}{ 1}\cdot \gy{n-1}{ 1}\cdot \gy{n-3}{ 2}\\
&& 
 \,-\,\gy{n}{ 1}\cdot \gy{n-4}{ 2}\,-\,\gx{n}{ 1}\cdot \gy{n-1}{ 1}\cdot \gy{n-4}{ 2}\,+\, \gx{n}{ 2} \cdot\gy{n-2}{ 2}\\
&& 
\,-\,\gx{n}{ 2}\cdot \gy{2}{ 1}\cdot \gy{n-3}{ 2}
\,-\,\gx{n}{ 1}\cdot \gy{n-3}{ 2}\,+\, \gx{n}{ 2}\cdot \gy{2}{ 1}\cdot \gy{n-4}{ 2}\\
&& 
\,-\,\gx{n}{ 2}\cdot \gy{n-4}{ 2}\,+\, \gx{n}{ 1}\cdot \gy{n-4}{ 2}. 
\end{eqnarray*}
For instance, consider the case where $n=7$.
Then the  Betti numbers   of  
the $12$-dimensional Frobenius incidence variety 
$X[r,s]_{2}^2$ 
are  as follows, where $q:=rs$: 
$$
\begin{array}{ll}
b_{{0}}=b_{{24}}:&1,\\
b_{{2}}=b_{{22}}:&2,\\
b_{{4}}=b_{{20}}:&5,\\
b_{{6}}=b_{{18}}:&{q}^{6}+{q}^{5}+{q}^{4}+{q}^{3}+{q}^{2}+q+8,\\
b_{{8}}=b_{{16}}:&2\,({q}^{6}+{q}^{5}+{q}^{4}+{q}^{3}+{q}^{2}+q)\,+12,\\
b_{{10}}=b_{{14}}:&3\,({q}^{6}+{q}^{5}+{q}^{4}+{q}^{3}+{q}^{2}+q)\,+14,\\
b_{{12}}:\phantom{=b_{{12}}}&{q}^{10}+{q}^{9}+2\,{q}^{8}+2\,{q}^{7}+6\,{q}^{6}+6\,{q}^{5}+6\,{q}^{4}+5\,{q}^{3}+5\,{q}^{2}+4\,q+16.
\end{array}
$$
\end{example}
\begin{remark}
The fact that $\Nlc(x,y)$ should be palindromic with respect to $y$ helps us in checking the computation
of $\Nlc(x,y)$.
\end{remark}
As a simple corollary of Propositions~\ref{prop:Ttau}~and~\ref{prop:tau},
we obtain the following.
Let $\kappa$ denote the function field of $X$.
By the generic point of $X$, we mean 
 the pair $(L_\eta, M_\eta)$ of $\kappa$-rational linear subspaces 
corresponding to
$$
\Spec\kappa\;\;\to\;\; X\;\;\inj\;\; \Gr_{n,l}\times\Gr_n^c,
$$
where $\Spec\kappa\to X$ is  the canonical morphism.
\begin{proposition}\label{prop:generic}
Let $(L_\eta, M_\eta)$ be the generic point of $X$.
\par
{\rm (1)}
If $2l+c\ge n$, then $L_\eta+ L_\eta\sp{rs} =M_\eta\sp{s}$.
If $2l+c\le n$, then the projection $X\to \Gr_{n,l}$ is surjective and hence $L_\eta+ L_\eta\sp{rs}$
is of dimension $2l$.
\par
{\rm (2)}
If $l+2c\ge n$, then $M_\eta\cap M_\eta\sp{rs} =L_\eta\sp{r}$.
If $l+2c\le n$, then the projection $X\to \Gr_n^c$ is surjective and hence $M_\eta\cap M_\eta\sp{rs} $
is of dimension $n-2c$.
\end{proposition}
\begin{proof}
We put $q:=rs$ again.
The function $d_L: (L,M)\mapsto \dim (L+L\sp{q})$ is lower semi-continuous 
and bounded by  $n-c$ from above on $X$.
If  $2l+c\ge n$,
then $\cond(l, 2l-n+c)$ is true and $\delta(l, 2l-n+c)>0$.
Therefore the set
$\Tq_{l, 2l-n+c} (\F_{q^\nu})$
is  non-empty for a sufficiently large $\nu$.
Hence $d_L$ attains $n-c$ on $X$,
and thus $\dim (L_\eta+L_\eta\sp{q}) =n-c$.
Therefore we have $L_\eta+L_\eta\sp{q}=M_\eta\sp{s}$.
Let $\kappa_\gamma$ denote the function field of $\Gr_{n,l}$.
If $2l+c\le n$, 
then the generic point $L_\gamma\in \Gr_{n,l}(\kappa_\gamma)$
satisfies  
$\dim (L_\gamma+ L_\gamma\sp{q})=2l\le n-c$.
There exists a $\kappa_\gamma$-valued point 
$N_\gamma\in \Gr_{n}^c(\kappa_\gamma)$
such that 
$L_\gamma+ L_\gamma\sp{q}\subset N_\gamma$,
and hence
$(L_\gamma, N_\gamma\sp{1/s})$ is a $\kappa_\gamma\sp{1/s}$-valued point of
 $X$.
Thus the assertion (1) is proved.
\par
The assertion (2) is proved in the dual way.
\end{proof}
\section{Unirationality}\label{sec:unirational}
In this section, we prove Theorem~\ref{thm:unirational}.
\par
\medskip
Note that the purely inseparable morphisms
$$
\Fr\spar{s}\times\id\;:\; \Gr_{V, l}\times \Gr_V^c\to \Gr_{V, l}\times \Gr_V^c
\quand 
\id \times \Fr\spar{r} \;:\; \Gr_{V, l}\times \Gr_V^c\to \Gr_{V, l}\times \Gr_V^c
$$
induce purely inseparable surjective morphisms
$$
X[rs,1]\to X[r,s]\quand X[1, rs]\to X[r,s]
$$
defined over $\F_p$.
Hence it is enough to prove  that 
$X[q,1]$ is purely-inseparably unirational over $\F_p$
for any power $q$ of $p$.
We prove this fact
by induction on $2l+c$.
\par
Suppose that $2l+c\le n$.
We show that $X[q,1]$ is rational over $\F_p$,
and hence, \emph{a fortiori},
purely-inseparably unirational over $\F_p$.
For the generic point $(L_\eta, M_\eta)$ of $X[q,1]$, we have
$\dim (L_\eta+L_\eta^q)=2l$ by Proposition~\ref{prop:generic}(1).
We fix an $\F_p$-rational linear subspace $K\subset V$ of dimension $2l$.
Then there exist a non-empty open subset $\UUU$ of $\Gr_{V, l}$ and a morphism 
$$
\alpha : \UUU\to \GL(V)
$$
defined over $\F_p$ such that,
for any field $F$ of characteristic $p$ and any $F$-valued point $L\in \UUU(F)$, we have
$\dim (L+L^q)=2l$ and $\alpha (L)\in \GL(V_F)$ induces an isomorphism
from $K\tensor F$ to $L+L^q$.
Let $M/K\mapsto M$ denote the natural embedding $\Gr_{V/K}^c\inj \Gr_{V}^c$.
Then the morphism 
$$
\UUU\times \Gr_{V/K}^c\;\;\to\;\; \Gr_{V, l}\times\Gr_V^c
$$
given by $(L, M/K)\mapsto (L, \alpha(L)(M))$
is a birational map  defined over $\F_p$ from the rational variety $\UUU\times \Gr_{V/K}^c$
to  $X[q,1]$,
with the  inverse rational map being given by
$$
(L, M) \mapsto (L, \alpha(L)\inv(M)/K).
$$
\par
Suppose that $2l+c> n$.
We put
$$
l\sprime:=2l+c-n
\quand
c\sprime:=n-l.
$$
Then we have $l\sprime>0$, $c\sprime>0$
and $l\sprime+c\sprime<n$.
We show that 
 $X:=X[q, 1]^c_l$ is birational over $\F_p$ to
$X\sprime:=X[1, q]^{c\sprime}_{l\sprime}$.
We denote by $\kappa$ and $\kappa\sprime$ 
the function fields of $X$ and  $X\sprime$, respectively.
Note that 
\begin{eqnarray*}
X(F) &=&\set{(L,M)\in\Gr_{n,l}(F)\times \Gr_n^c (F)}{L+L^q \subset M},\\
X\sprime(F) &=&\set{(L\sprime,M\sprime)\in\Gr_{n,2l+c-n}(F)\times \Gr_{n,l} (F)}%
{L\sprime\subset M\sprime\cap M\sp{\prime q}}.
\end{eqnarray*}
By Proposition~\ref{prop:generic}(1),
the generic point $(L_\eta, M_\eta)$ of $X$ satisfies 
$L_\eta+L_\eta^q=M_\eta$,
and hence $\dim (L_\eta\cap L_\eta^q)=2l+c-n$.
Therefore we have
$(L_\eta\cap L_\eta^q, L_\eta)\in X\sprime(\kappa)$,
and hence 
$$ 
(L, M)\mapsto (L\cap L^q, L)
$$
defines  a rational map $\rho : X\ratmap X\sprime$
defined over $\F_p$.
On the other hand, we have
$$
l\sprime+2  c\sprime=n+c>n.
$$
By Proposition~\ref{prop:generic}(2),
the generic point $(L\sprime_\eta, M\sprime_\eta)$ of $X\sprime$ satisfies 
$M\sprime_\eta\cap M\sp{\prime q}_\eta=L\sprime_\eta$,
and hence 
$$
\dim (M\sprime_\eta+  M\sp{\prime q}_\eta)=2(n-c\sprime)-l\sprime=n-c.
$$
Therefore we have
$(M\sprime_\eta, M\sprime_\eta+  M\sp{\prime q}_\eta)\in X(\kappa\sprime)$,
and hence
$$
(L\sprime, M\sprime)\mapsto (M\sprime, M\sprime+M\sp{\prime q})
$$
defines a rational map $\rho\sprime: X\sprime\ratmap X$
defined over $\F_p$.
Note that
$\rho\sprime(\rho (L_\eta, M_\eta))$ is defined and equal to $(L_\eta, M_\eta)$.
Note also that 
$\rho(\rho\sprime (L\sprime_\eta, M\sprime_\eta))$ is defined and equal to 
$(L\sprime_\eta, M\sprime_\eta)$.
Hence $X$ and $X\sprime$ are birational over $\F_p$.
Since
$$
2l\sprime +c\sprime=2l+c-(n-l-c)<2l+c,
$$
 the induction hypothesis implies that $X\sprime$ is purely-inseparably unirational
 over $\F_p$.
 Therefore $X$ is also purely-inseparably unirational over $\F_p$.
 \qed
\begin{remark}\label{rem:unirational}
We have established the facts that
 $X[q, 1]^c_l$ is rational over $\F_p$ for $2l+c\le n$, and that
 $X[1, q]^c_l$ is rational over $\F_p$ for $l+2c\le n$.
\end{remark}
\section{Intersection pairing}\label{sec:intnumb}
In this section,
we calculate the intersections of the subvarieties $\Sigma_\Lambda$ 
defined by~\eqref{eq:defSigma} in the Chow ring  of  $X=X[r,s]^c_l$.
\par
\medskip
For a smooth projective variety $Y$ of dimension  $m$,
we denote by $A^k(Y)=A_{m-k} (Y)$ the Chow group of rational equivalence classes
of algebraic cycles on $Y$ with codimension $k$
defined over an algebraic closure of the base field,
and by $A(Y)=\bigoplus A^k(Y)$ the Chow ring of $Y$.

\par
\medskip
In order to state our main result, 
we need to define a homomorphism $\tlchi$.
Let $W$ be a $w$-dimensional linear space,
and let
$$
S_{W, l}\to \Gr_{W, l}=\Gr_{w, l}
$$
denote  the universal subbundle of $W\tensor \OOO$ over $\Gr_{W, l}$.
Let $x_1, \dots, x_l$ be the formal Chern roots of the total Chern class
$c(S_{W, l}\dual)$ of the dual vector bundle $S_{W, l}\dual$:
$$
c(S_{W, l}\dual)=(1+x_1)\cdot \cdots \cdot(1+x_l).
$$
Then we have a natural homomorphism
\begin{equation*}\label{eq:gammaW}
\symchi_{w,l}\;:\; \Z[[x_1, \dots, x_l]]^{\SSSS_l} \to A(\Gr_{W, l})
\end{equation*}
from the ring of symmetric power series in variables $x_1, \dots, x_l$
with $\Z$-coefficients 
to the Chow ring $A(\Gr_{W, l})$ of $\Gr_{W, l}$.
Let $U$ be a $u$-dimensional linear space,
and let
$$
Q_{U}^c\to \Gr_{U}^c=\Gr_{u}^c
$$
denote  the universal quotient bundle  of $U\tensor \OOO$ over $\Gr_{U}^c$.
Let $y_1, \dots, y_c$ be the formal Chern roots of the total Chern class
$c(Q_{U}^c)$:
$$
c(Q_{U}^c)=(1+y_1)\cdot\cdots \cdot(1+y_c).
$$
Then we have a natural homomorphism
\begin{equation*}\label{eq:gammaU}
\symchi_{u}^c\;:\;
\Z[[y_1, \dots, y_c]]^{\SSSS_c} \to A(\Gr_{U}^c).
\end{equation*}
Composing  $\symchi_{w,l}\tensor \symchi_u^c$  and
the natural homomorphism 
$$
A(\Gr_{W, l})\tensor A(\Gr_{U}^c)\to A(\Gr_{w, l}\times \Gr_{u}^c),
$$
we obtain a homomorphism
$$
\tlchi\;\;:\;\; \Z[[x_1, \dots, x_l, y_1, \dots, y_c]]^{\SSSS_l\times \SSSS_c} \to 
A(\Gr_{w, l}\times \Gr_{u}^c)
$$
from the ring of ${\SSSS_l\times \SSSS_c}$-symmetric power series to $A(\Gr_{w, l}\times \Gr_{u}^c)$.
\par
\medskip
Let $\Lambda$ and $\Lambda\sprime$ be $\F_{rs}$-rational linear subspaces
of $V_{\F}:=V\tensor \F_{rs}$.
We consider the intersection of the subvarieties $\Sigma_{\Lambda}$ and $\Sigma_{\Lambda\sprime}$
of $X$ in $A(X)$.
We put 
$$
\intdim:=\dim (\Lambda\cap\Lambda\sprime)
\quand
\sumcodim:=n-\dim (\Lambda+\Lambda\sprime).
$$
Then we have
$$
e:=\dim\Sigma_{\Lambda}+\dim\Sigma_{\Lambda\sprime}-\dim X=(l-c)(c-l+m-k),
$$
and the intersection 
of $\Sigma_{\Lambda}$ and $\Sigma_{\Lambda\sprime}$ 
in $A(X)$ is an element of $A_e(X)$.
We put
$$
\Upsilon:=\Lambda\cap \Lambda\sprime
\quand \Theta:=V_{\F}/ (\Lambda+\Lambda\sprime)\sp{r}.
$$
Since
$\Sigma_{\Lambda}=\Gr_{\Lambda, l}\times \Gr_{V_{\F}/\Lambda\sp{r}}^c$
and
$\Sigma_{\Lambda\sprime}=\Gr_{\Lambda\sprime, l}\times \Gr_{V_{\F}/\Lambda\sp{\prime r}}^c$,
the scheme-theoretic intersection of 
$\Sigma_{\Lambda}$ and $\Sigma_{\Lambda\sprime}$  
is the smooth subscheme
$$
\Gamma:=\Gr_{\Upsilon, l}\times \Gr_{\Theta}^c \;\;\cong\;\; \Gr_{\intdim, l}\times \Gr_{\sumcodim}^c.
$$
Then the intersection of $\Sigma_{\Lambda}$ and $\Sigma_{\Lambda\sprime}$ in
$A_e(X)$ is localized  
in $A_e (\Gamma)=A^d (\Gamma)$, where 
$$
d:=\dim\Gamma-e=\sumcodim  l+\intdim c-2lc.
$$
The following is the main result of this section:
\begin{theorem}\label{thm:intnumb}
Let $\Lambda$ and $\Lambda\sprime$ be as above.
Then the intersection  of $\Sigma_{\Lambda}$ and $\Sigma_{\Lambda\sprime}$ in $A(X)$
is equal to the image of the codimension $d$ part of 
$$
\tlchi (f)\in A(\Gamma)
$$
by the push-forward homomorphism $A_e(\Gamma)\to A_e(X)$, 
where $f$ is the ${\SSSS_l\times \SSSS_c}$-symmetric power series
$$
\frac{\prod_{i=1}^l(1+x_i)^{\sumcodim}\cdot \prod_{j=1}^c(1+y_j)^\intdim}{
\prod_{i=1}^l \prod_{j=1}^c (1+rx_i+y_j) (1+x_i+s y_j)
},
$$
and $\tlchi$ is 
the homomorphism to $A(\Gamma)=A(\Gr_{\intdim, l}\times \Gr_{\sumcodim}^c)$ defined above.
\end{theorem}
\begin{proof}
We denote by $\TTT(Y)\to Y$ the tangent bundle of a smooth variety $Y$.
Note that the tangent bundle of a Grassmannian variety
is the tensor product of the dual of the universal subbundle and the universal quotient bundle.
\par
Let $W$ and $W\sprime$ be linear subspaces of $V_{\F}$
with dimension $w$ and $w\sprime$, respectively, 
such that
$W\subset W\sprime$,
and let $i:\Gr_{W, l}\inj \Gr_{W\sprime, l}$
denote the natural embedding.
For the universal subbundles 
$S_{W, l}\to \Gr_{W, l}$ and $S_{W\sprime, l}\to \Gr_{W\sprime, l}$,
we have
\begin{equation}\label{eq:iS}
i^*S_{W\sprime, l}\dual =S_{W, l}\dual.
\end{equation}
We denote by
$Q_W^{w-l}\to \Gr_{W, l}$ and $Q_{W\sprime}^{w\sprime-l}\to \Gr_{W\sprime, l}$
the universal quotient bundles.
Then we have an exact sequence
$$
0\;\to\; Q_W^{w-l}\;\to\; i^* Q_{W\sprime}^{w\sprime-l} \;\to\; W\sprime/W\tensor \OOO \;\to\; 0
$$
of vector bundles over $\Gr_{W, l}$.
Therefore we have the following  equality in  $A(\Gr_{W, l})$:
\begin{equation}\label{eq:cTanGl}
i^* c(\TTT ( \Gr_{W\sprime, l}))=c(\TTT ( \Gr_{W, l}))\cdot c(S_{W, l}\dual)^{w\sprime-w}.
\end{equation}
Let $V_{\F}\to U=V_{\F}/K$ and $V_{\F}\to U\sprime=V_{\F}/K\sprime$ 
be linear quotient spaces of $V_{\F}$
 such that 
$K\sprime\subset K$.
We put $u:=\dim U$ and  $u\sprime:=\dim U\sprime$.
Let $j:\Gr_{U}^c\inj \Gr_{U\sprime}^c$
denote the natural embedding.
For the universal quotient bundles 
$Q_U^c\to \Gr_U^c$ and $Q_{U\sprime}^c\to \Gr_{U\sprime}^c$,
we have
\begin{equation}\label{eq:jQ}
j^*Q_{U\sprime}^c =Q_U^c.
\end{equation}
\erase{
We denote by
$S_{U, u-c}\to \Gr_U^c$ and by $S_{U\sprime, u\sprime-c}\to \Gr_{U\sprime}^c$
the universal subbundles.
Then we have an exact sequence
$$
0\;\to\;S_{U, u-c}\dual  \;\to\; j^* S_{U\sprime, u\sprime-c}\dual \;\to\; 
(K/K\sprime)\dual\tensor\OOO  \;\to\; 0
$$
of vector bundles over $\Gr_U^c$.
Therefore we have the following  equality in $A(\Gr_U^c)$:
}
By the argument dual to the above, 
we obtain the following  equality in $A(\Gr_U^c)$:
\begin{equation}\label{eq:cTanGc}
j^* c(\TTT ( \Gr_{U\sprime}^c))=c(\TTT ( \Gr_U^c))\cdot c(Q_U^c)^{u\sprime-u}.
\end{equation}
\par
\medskip
We consider the vector bundle
$$
\XXX:=\frac{\TTT(X)|_\Gamma }{\TTT(\Sigma_{\Lambda})|_\Gamma + \TTT(\Sigma_{\Lambda\sprime})|_\Gamma}
$$
of rank $d$ over $\Gamma$.
By the excess intersection formula~\cite[p.~102]{MR1644323},
the intersection of $\Sigma_{\Lambda}$ and $\Sigma_{\Lambda\sprime}$
in $A(X)$ is equal to the image of 
the top Chern class $c_{d}(\XXX)\in A^{d} (\Gamma)=A_e(\Gamma)$
of $\XXX$
 by the push-forward homomorphism $A_e(\Gamma)\to A_e(X)$. 
Hence it is enough to show that 
the total Chern class $c(\XXX)\in A(\Gamma)$ of $\XXX$
is equal to $\tlchi (f)$.
From the exact  sequence
$$
0
\;\to\;
\TTT(\Gamma)
\;\to\;
\TTT(\Sigma_{\Lambda})|_\Gamma \oplus \TTT(\Sigma_{\Lambda\sprime})|_\Gamma 
\;\to\;
\TTT(\Sigma_{\Lambda})|_\Gamma + \TTT(\Sigma_{\Lambda\sprime})|_\Gamma 
\;\to\;
 0
$$
and Proposition~\ref{prop:FFF}, we have
$$
c(\XXX)=
\frac{c(\TTT(\Gr_{V, l}\times \Gr_{V}^c)|_\Gamma)\cdot c(\TTT(\Gamma)) }
{c(\FFF|_\Gamma)\cdot c(\TTT(\Sigma_{\Lambda})|_\Gamma)\cdot c(\TTT(\Sigma_{\Lambda\sprime})|_\Gamma) }.
$$
We put
$$
\lambda:=\dim\Lambda\quand \lambda\sprime:=\dim\Lambda\sprime.
$$
By~\eqref{eq:iS}-\eqref{eq:cTanGc}, 
we have the following equalities in 
$A(\Gamma)$:
\begin{eqnarray*}
c(\TTT(\Gr_{V, l}\times \Gr_{V}^c)|_\Gamma)&=& c(\TTT(\Gamma)) \cdot 
\left(
c(S_{\Upsilon, l}\dual)^{n-\intdim}\,\tensor c(Q_{\Theta}^c)^{n-\sumcodim}
\right),\\
c(\TTT(\Sigma_{\Lambda})|_\Gamma)&=& c(\TTT(\Gamma)) \cdot 
\left(
c(S_{\Upsilon, l}\dual)^{\lambda-\intdim}\,\tensor c(Q_{\Theta}^c)^{n-\sumcodim-\lambda}
\right),\\
c(\TTT(\Sigma_{\Lambda\sprime})|_\Gamma)&=& c(\TTT(\Gamma)) \cdot 
\left(
c(S_{\Upsilon, l}\dual)^{\lambda\sprime-\intdim}\tensor c(Q_{\Theta}^c)^{n-\sumcodim-\lambda\sprime}
\right).
\end{eqnarray*}
Here $c(S_{\Upsilon, l}\dual)^{\mu}\tensor c(Q_{\Theta}^c)^{\nu} \in A(\Gr_{\Upsilon, l})\tensor A( \Gr_{\Theta}^c)$
is identified with  its image in $A(\Gamma)=A(\Gr_{\Upsilon, l}\times \Gr_{\Theta}^c)$.
Since $\lambda+\lambda\sprime=\intdim+n-\sumcodim$,
we have
$$
c(\XXX)=c(\FFF|_\Gamma)\inv \cdot 
\left( 
c(S_{\Upsilon, l}\dual)^{\sumcodim} 
\tensor
c(Q_{\Theta}^c)^{\intdim}
\right).
$$
By the definition~\eqref{eq:defFFF} of $\FFF$ and~\eqref{eq:iS} and~\eqref{eq:jQ}, we have 
$$
\FFF|_\Gamma
\;=\;
({ \Fr^{(r)\,*}}S_{\Upsilon, l}\dual\tensor Q_{\Theta}^c)
\,\oplus\,
(S_{\Upsilon, l}\dual\tensor {\Fr^{(s)\,*}}Q_{\Theta}^c),
$$
where $\pr^*$ is omitted.
Note that 
$$
c({ \Fr^{(r)\,*}} S_{\Upsilon, l}\dual)=\prod_{i=1}^l  (1+r x_i)\;\;\textrm{in}\;\; A(\Gr_{\Upsilon, l})
\quand 
c({\Fr^{(s)\,*}} Q_{\Theta}^c)=\prod_{j=1}^c (1+s y_j)\;\;\textrm{in}\;\; A(\Gr_{\Theta}^c).
$$
Hence  $c(\FFF|_\Gamma)\in A(\Gamma)$ is the image of
$$
\prod_{i=1}^l \prod_{j=1}^c (1+rx_i+y_j) (1+x_i+s y_j)\in 
\Z[[x_1, \dots, x_l, y_1, \dots, y_c]]^{\SSSS_l\times \SSSS_c}
$$
by the homomorphism $\tlchi $.
Therefore we have $c(\XXX)=\tlchi (f)$ in $A(\Gamma)$.
\end{proof}
When  $l=c$, 
the intersection number of $\Sigma_{\Lambda}$ and $\Sigma_{\Lambda\sprime}$ in $X[r,s]^l_l$
is defined  and equal to the degree of the $A_0$-component of $\tlchi(f)\in A(\Gamma)$.
When $l=c=1$,
we have $S_{\Upsilon, 1}\dual=\OOO(1)$ on $\Gr_{\Upsilon, 1}\cong \P^{\intdim-1}$
and $Q_{\Theta}^1=\OOO(1)$ on $\Gr_{\Theta}^1\cong \P^{\sumcodim-1}$.
Therefore we obtain the following:
\begin{corollary}\label{cor:intnumb2}
Suppose that $l=c=1$, and 
let $\Lambda$ and $\Lambda\sprime$ be  as above.
Then the intersection number of  $\Sigma_{\Lambda}$ and $\Sigma_{\Lambda\sprime}$ in 
$X[r,s]_{1}^1$
is equal to the coefficient of $x^{\intdim-1}y^{\sumcodim-1}$ in the power series
$$
(1+rx+y)\inv (1+x+s y)\inv (1+x)^{\sumcodim} (1+y)^{\intdim}
\;\; \in \;\;\Z[[x, y]].
$$
\end{corollary}
\section{The lattice $\latN(X)$}\label{sec:latticeN}
In this section,
we treat the case $l=c=1$.
We put $X:=X[r,s]^1_1$ throughout this section.
\par
\medskip
We denote by $\disc \latN(X)$ the discriminant of $\latN(X)$.
We put
$$
f(n):=\frac{(rs)^n-1}{rs-1}=|\P^{n-1} (\F_{rs})|.
$$
\begin{theorem}\label{thm:NNN}
The  rank of the lattice $\latN(X)$ associated with  
$X=X[r,s]^1_1$ is  equal to
the middle Betti number $b_{2(n-2)}(X)$ of $X$.
If $n>3$, 
 $\disc \latN(X)$ is a divisor of 
$\min(r, s)^{(n-2)(f(n)-1)}$,
while 
if $n=3$, 
$\disc \latN(X)$ is a divisor of $ \min(r, s)^{f(3)+1}$.
\end{theorem}
For the proof of Theorem~\ref{thm:NNN},  we fix notation.
We write $[\Sigma_\Lambda]\in \latN(X)$ and  $[h_i]\in \latN(X)$
for 
the rational equivalence classes of the algebraic cycles  $\Sigma_\Lambda$ and $h_i$ modulo $\latNtl(X)\sperp$.
\par
\medskip
Let $\Pset{0}$ denote the set of $\F_{rs}$-rational points of $\P_*(V)$,
whose cardinality is $f(n)$.
For a positive integer $k<n$, 
let $\Lset{k}$ denote the set of $k$-dimensional $\F_{rs}$-rational linear subspaces  of 
$V_{\F}:=V\tensor\F_{rs}$.
For $\Lambda\in \Lset{k}$,
we denote by  $\P_*(\Lambda)$  
the corresponding $(k-1)$-dimensional 
projective linear subspace of $\P_*(V)$ over $\F_{rs}$,
and put
$$
S(\Lambda):=\set{P\in \Pset{0}}{ P\in \P_*(\Lambda)}.
$$
For $P\in \Pset{0}$,
let $\ell(P)\in \Lset{1}$ denote  the corresponding 
 $\F_{rs}$-rational linear subspace of dimension $1$.
\par
\medskip
We calculate the intersection numbers 
of the classes  $[h_i]$ and $[\Sigma_{\Lambda}]$ in $\latN(X)$.
By Corollary~\ref{cor:intnumb2}, for $P\in \Pset{0}$ and $\Lambda\in \Lset{k}$, we have
\begin{equation}\label{eq:SigmaSigma}
([\Sigma_{\Lambda}], [\Sigma_{\ell(P)}])=\begin{cases}
(-s)^{n-k -1} &\textrm{if $P\in S(\Lambda)$,} \\
0 & \textrm{otherwise.}
\end{cases}
\end{equation}
For $\Lambda\in \Lset{k}$, 
the subvariety   $\Sigma_{\Lambda}$ is a Cartesian product of 
$\P_*(\Lambda)\subset \P_*(V)$ and $\P^*(V/\Lambda\sp{r})
\subset \P^*(V)$
with  $\dim \P_*(\Lambda)=k-1$ and
 $\dim \P^*(V/\Lambda\sp{r})=n-1 -k$.
Hence we have
\begin{equation}\label{eq:hSigma}
([h_i],  [\Sigma_{\Lambda}])=
\begin{cases}
1 & \textrm{if $i+k=n$}, \\
0 & \textrm{otherwise}. \\
\end{cases}
\end{equation}
Recall from Proposition~\ref{prop:FFF} that 
$X\subset \P_*(V)\times \P^*(V)$ is
a subvariety of codimension $2$ 
defined as  the zero locus of the section $\tilde\gamma$ of 
the vector bundle $\OOO(r, 1)\oplus \OOO(1, s)$ of rank $2$.
Hence the intersection numbers of 
the classes $[h_i]$  are 
\begin{equation}\label{eq:hh}
([h_i], [ h_j])=
\begin{cases}
s & \textrm{if $i+j=n-1$}, \\
1+rs & \textrm{if $i+j=n$}, \\
r & \textrm{if $i+j=n+1$}, \\
0 & \textrm{otherwise}. \\
\end{cases}
\end{equation}
We fix a point $\baseptP \in \Pset{0}$, and consider the following four submodules 
of $\latN(X)$:
\begin{eqnarray*}
\HHH &:= & \gen{\;[h_1], \dots, [h_{n-1}]\;}, \\
\MMM &:= & \genset{ [\Sigma_{\ell(P)}] }{ P\in \Pset{0} }, \\
\MMM_0 &:= & \genset{ [\Sigma_{\ell(P)}] }{ P\in \Pset{0}, P\ne \baseptP  }, \\
\MMM_D &:= & \genset{ [D_{\ell(P)}]}{ P\in \Pset{0}, P\ne \baseptP  }, \quad
\textrm{where $[D_{\ell(P)}]:= [\Sigma_{\ell(P)}]-[\Sigma_{\ell(\baseptP )}]$.} 
\end{eqnarray*}
Here $\gen{v_1, \dots, v_N}$ denotes the submodule generated by $v_1, \dots, v_N$. 
\par
\medskip
The following is elementary:
\begin{lemma}\label{lem:detA}
Let $m$ be an integer $\ge 3$,
and let $u, v, t$ be indeterminants.
Consider  the  $m\times m$ matrix
$A(m, u, v, t)=(a_{ij})_{1\le i, j\le m}$   defined by
$$
a_{ij}:=\begin{cases}
u  &\textrm{if $i+j=m$},\\
1+uv  &\textrm{if $i+j=m+1$},\\
v  &\textrm{if $i+j=m+2$},\\
t  &\textrm{if $i=j=m$},\\
0 &\textrm{otherwise}.\\
\end{cases}
$$
Then we have
$$
\det A(m, u, v, t)=(-1)^{[m/2]}\left(\frac{(uv)^{m+1}-1}{uv-1}+(-u)^{m-1} t\right),
$$
where
$[m/2]$ denotes the integer part of $m/2$.
\qed
\end{lemma}
\begin{proof}[Proof of Theorem~\ref{thm:NNN}]
By the duality,
we can assume that
$s\le r$.
\par
If the cohomology class of $x\in \latNtl  (X)$ is zero,
then $x$ is obviously contained 
$\latNtl  (X)\sperp$.
Hence, by~Example~\ref{example:Nn11second},  the rank of $\latN(X)$ is at most
$$
b:=b_{2(n-2)}(X)=n+f(n)-2=(n-1)+(|\Pset{0}|-1).
$$
\par
\medskip
First  assume that $n>3$. 
We
show that 
$\latN(X)$ is of rank $b$,
and that its  discriminant divides  $s^{(n-2)(f(n)-1)}$.
Consider the submodule $\HHH+\MMM_0$ of $\latN(X)$ generated  by the $b$ classes 
\begin{equation}\label{eq:blist}
[h_1], \dots, [h_{n-1}], \;\; [\Sigma_{\ell(P)}]\;\; (P\in \Pset{0},\, P\ne \baseptP ).
\end{equation}
By~\eqref{eq:SigmaSigma}, \eqref{eq:hSigma} and \eqref{eq:hh},
the intersection matrix of these classes is
$$
\wt{A}:=
\left[
\begin{array}{c|c}
A_{\HHH} &\begin{array}{ccc}   &\\ &O &  \\ 1 &\cdots & 1 \end{array}\\
\hline
\begin{array}{ccc} &&1 \\ &O\phantom{a}&\vdots\mystruthd{20pt}{15pt}    \\ &&1 \end{array} 
& (-s)^{n-2} I 
\end{array}
\right],
$$
where $A_{\HHH}:=A(n-1, s, r, 0)$ is the intersection matrix of the classes $[h_i]$ and 
$I$ is the identity matrix of size ${f(n)-1}$.
By Lemma~\ref{lem:detA}, we have
\begin{eqnarray*}
\det \wt{A}&=&
\det A(n-1, s, r, t_0)\cdot \det\, ((-s)^{n-2} I)\mystruthd{0pt}{7pt}
\\
&=&(-1)^{[(n-1)/2]} \cdot (-s)^{(n-2) (f(n)-1)}\;\;\ne\;\; 0,
\end{eqnarray*}
where $t_0:=-(f(n)-1)/(-s)^{n-2}$.
Thus  $\HHH+ \MMM_0$ is a lattice of rank $b$ 
with the basis~\eqref{eq:blist}.
Since $\rank \latN(X)\le b$,
 we conclude that $\rank \latN(X)=b$ and that
 $\latN(X)$ contains $\HHH+ \MMM_0$ as a sublattice of finite index.
Therefore 
 $\disc \latN(X)$ is a divisor  of
 $\disc (\HHH+ \MMM_0)=\pm s^{(n-2) (f(n)-1)}$.
\par
For the case $n=3$,
we consider the submodule 
$\ang{[h_1]}+\MMM$.
The intersection matrix of the generators
$[h_1]$ and $[\Sigma_{\ell(P)}]$  $(P\in \Pset{0})$
of this submodule  is the diagonal matrix of size $b=f(3)+1$
with diagonal components
$s, -s, \dots, -s$.
Consequently, 
 $\ang{[h_1]}+\MMM$ is a lattice of rank $b$ with the discriminant $\pm  s^{f(3)+1}$.
Hence $\latN (X)$ is a lattice of rank $b$ containing
$\ang{[h_1]}+\MMM$ as a sublattice of finite index,
and   $\disc \latN (X)$ is a divisor of $s^{f(3)+1}$.
\end{proof}
\begin{proof}[Proof of Corollary~\ref{cor:cohring}]
By Proposition~\ref{prop:FFF},  
the  subvariety $X\subset \P_*(V)\times \P^*(V)$ is
a smooth complete intersection of very ample divisors 
$D_1\in |\OOO(r, 1)|$ and $D_2\in |\OOO(1, s)|$.
Hence, by Lefschetz hyperplane section theorem of Deligne~\cite{MR601520}
(see also~\cite{MR0472817}),
the inclusion of $X$ into $ \P_*(V)\times \P^*(V)$ induces
 isomorphisms of $l$-adic cohomology groups in degree $<\dim X$.
 On the other hand, 
Theorem~\ref{thm:NNN} implies that  the cycle map 
induces an isomorphism from $\latN(X)\tensor \Q_l$
to the  middle $l$-adic cohomology group of $X$.
\end{proof}
\begin{remark}\label{rem:unimod}
Theorem~\ref{thm:NNN} implies that, if $r=1$ or $s=1$, then 
$\latN(X[r,s]^1_1)$ is unimodular.
Recall from 
Remark~\ref{rem:unirational} that,
if $r=1$ or $s=1$, then $X[r,s]^1_1$ is a rational variety.
\end{remark}
Next we prove Theorem~\ref{thm:NNNprim} on  the primitive part $\Nprim(X)$ of $\latN(X)$.
\begin{proof}[Proof of Theorem~\ref{thm:NNNprim}]
We use the notation in the proof of Theorem~\ref{thm:NNN}.
Since 
$$
\det A_{\HHH}=\det A(n-1, s, r, 0)=\pm f(n)\ne 0,
$$
the submodule $\HHH$ is a sublattice of $\latN(X)$ with rank $n-1$.
Therefore $\Nprim (X)=\HHH\sperp$ is also a sublattice with 
$$
\rank \Nprim (X)=b-(n-1)=f(n)-1.
$$
By~\eqref{eq:hSigma},
the classes
\begin{equation}\label{eq:Sigmadiff}
[\Sigma_\Lambda]-[\Sigma_{\Lambda\sprime}]
\qquad(\Lambda, \Lambda\sprime\in \Lset{k}, \,k=1, \dots, n-1)
\end{equation}
are contained in $\Nprim (X)$. 
In particular, we have
$\MMM_D\subset \Nprim (X)$.
By~\eqref{eq:SigmaSigma},
we have
$$
([D_{\ell(P)}],
[D_{\ell(P\sprime)}])=\begin{cases}
2(-s)^{n-2} & \textrm{if $P=P\sprime$, }\\
(-s)^{n-2} & \textrm{if $P\ne P\sprime$. }
\end{cases}
$$
Hence the intersection matrix $A_D$ of the
classes $[D_{\ell(P)}]$ ($P\in \Pset{0}, P\ne \baseptP$) 
is non-degenerate.
Therefore $\MMM_D$ is of rank $f(n)-1$,
and we have 
\begin{equation}\label{eq:MDprim}
\MMM_D\tensor \Q=\Nprim (X)\tensor \Q.
\end{equation}
The symmetric  matrix $A_D$ multiplied by $(-1)^n$ defines a  positive-definite
quadratic form.
Hence 
$[-1]^{n}\Nprim (X)$
 is a positive-definite lattice.
 \end{proof}
\section{Dense lattices}\label{sec:dense}
In this section,
 we investigate  the case where
 $n=4$, $l=c=1$ and $p=r=s=2$,
 and prove Theorem~\ref{thm:Nprim}.
 We put $X:=X [2,2]_{1}^1$ 
 throughout this section.
 Note that $X$ is of dimension $4$.
 \par
\medskip
The \emph{minimal norm $\Nmin(L)$}  of a positive-definite lattice $L$ of rank $m$ is
the minimum of norms $x^2$ of non-zero vectors $x\in L$, and 
the \emph{normalized center density} $\delta (L)$ of  $L$ 
is defined  by
$$
\delta (L):=(\disc L)\sp{-1/2} \cdot (\Nmin(L)/4)\sp{m/2},
$$
where $\disc L$ is  the discriminant of $L$.
It is known that,
for each $m$, there exists a lattice $L$ 
such that  $\delta (L)$ exceeds
the Minkowski-Hlawka bound
$$
\zeta(m)\cdot 2\sp{-m+1}\cdot V\sb{m}\sp{-1},
$$
where $\zeta$ is the Riemann zeta function and 
$V\sb{m}$ is the volume of the $m$-dimensional unit ball.
(See~\cite[Chap.~VI]{MR0306130} or~\cite[Chap.~1]{MR1662447} for the Minkowski-Hlawka theorem.)
However the proof is not constructive.
\par
\medskip
We recall the notion of  \emph{dual lattices}.
Let $L$ be a lattice.
Then $L\tensor \Q$ is equipped with the $\Q$-valued symmetric bilinear form
that extends  the $\Z$-valued symmetric bilinear form on $L$.
We define the \emph{dual lattice $L\dual$} of $L$ by
$$
L\dual:=\set{x\in L\tensor \Q}{(x, y)\in \Z\;\;\textrm{for any\, $y\in L$}}.
$$
Then $L\dual$ is a  $\Z$-module 
containing $L$ as a submodule of finite index.
By definition, 
if $L_1$ and $L_2$ are sublattices of a lattice $L_3$ such that $L_1\subset L_2\tensor \Q$,
then $L_1$ is contained in  $L_2\dual$.
\par
\medskip
We use the notation of the previous section adapted to the present situation
$n=4$ and $p=r=s=2$.
Note that $\MMM$ is a lattice of rank $f(4)=|\Pset{0}|=85$
with the orthogonal basis $[\Sigma_{\ell(P)}]$ ($P\in \Pset{0}$).
Let 
$$
\NSigma(X)\;\subset\; \Nprim(X)
$$
be the submodule 
generated by the classes~\eqref{eq:Sigmadiff}. 
Since $\MMM_{D}\subset \NSigma(X)$, we have
$$
\MMM_D\tensor \Q=\NSigma(X)\tensor\Q=\Nprim (X)\tensor \Q
$$
by~\eqref{eq:MDprim}.
In particular, $\NSigma(X)$ is a lattice.
Since $\MMM_{D}\subset \MMM$, 
we have $\NSigma(X)\subset \MMM\tensor\Q$ in $\latN(X)\tensor \Q$.
We apply the above argument to  $L_1=\NSigma(X)$, $L_2=\MMM$, $L_3=\latN(X)$, and 
regard  $\NSigma(X)$ as embedded  in the dual lattice $\MMM\dual$.
\par
\medskip
Let $e_P$ $(P\in \Pset{0})$ be the  basis of $ \MMM\dual$
dual  to the orthogonal basis $[\Sigma_{\ell(P)}]$ of $\MMM$:
\begin{equation}\label{eq:sgnnMMM}
\MMM\dual:=\bigoplus_{P\in \Pset{0}} \Z e_P\;\cong\;\Z^{85}.
\end{equation}
We describe 
the submodule  $\NSigma\subset  \MMM\dual$
in a combinatorial way
using the projective geometry of $\Pset{0}=\P^{3}(\F_{4})$. 
We  put
$$
\Pset{k-1}:=\set{S(\Lambda)}{\Lambda\in \Lset{k}},
$$
which is a subset of the power set $2^{\Pset{0}}$ of $\Pset{0}$.
By~\eqref{eq:SigmaSigma},
the vector 
 $[\Sigma_{\ell(P)}] \in \MMM  \subset  \MMM\dual$
is equal  to $s^{n-2}\, e_P=4e_p$,  and 
hence we have $\MMM=s^{n-2}\,(\MMM\dual)=4\MMM\dual$.
Moreover 
the $\Q$-valued symmetric bilinear form on 
$\MMM\dual$ is  given by
$$
(e_P, e_{P\sprime})=
\begin{cases}
1/s^{n-2}=1/4 & \textrm{if $P=P\sprime$,}\\
0 & \textrm{if $P\ne P\sprime$}.
\end{cases}
$$
For $S\in 2^{\Pset{0}}$,
we put 
$$
v_S:=\sum_{P\in S} e_P\;\in\; \MMM\dual.
$$
By~\eqref{eq:SigmaSigma}, 
we see that 
$\NSigma$ is the submodule of $\MMM\dual$ generated by  
\begin{equation}\label{eq:combi}
s^{3-k} (v_S-v_{S\sprime})\qquad (S, S\sprime\in \Pset{k-1}, \;k=1, \dots, 3).
\end{equation}
\par
\medskip
Next we  introduce a code $\CCC$ over 
$$
R:=\Z/s^{n-1}\Z=\Z/8\Z
$$
and a  lattice $\MMM_{\CCC}$.
The  reduction homomorphism 
$ \MMM\dual \to  \MMM\dual\tensor R$ is denoted by $v\mapsto\bar v$.
Let $\CCC\subset  \MMM\dual\tensor R$ be
the image of $\NSigma\subset   \MMM\dual$ by $v\mapsto\bar v$.
Using~\eqref{eq:sgnnMMM}, we regard $\CCC$ as a submodule of $R^{f(4)}=R^{85}$, and 
consider $\CCC$ as an  $R$-code of length $85$.
Let $\MMM_{\CCC}\subset  \MMM\dual$ denote 
the pull-back of $\CCC$ by the reduction homomorphism.
Since
$$
\MMM_{\CCC}= (\NSigma)  + 8 (\MMM\dual)=(\NSigma)  + 2 (\MMM),
$$
the $\Q$-valued symmetric bilinear form on $\MMM\dual$ takes values in 
$\Z$ on $\MMM_{\CCC}$.
Therefore $\MMM_{\CCC}$ is a lattice.
\begin{theorem}\label{thm:MCCC}
The lattice  $\MMM_{\CCC}$ is an even  positive-definite lattice of rank $85$,
with discriminant $2^{20}$, and of minimal norm $8$. 
\end{theorem}
\begin{proof}
Since $\MMM_{\CCC}$ is the submodule generated by the vectors~\eqref{eq:combi} 
and
 $8e_P$ $(P\in \Pset{0})$
in $\MMM\dual$,
we can calculate the basis and the Gram matrix of $\MMM_{\CCC}$ by a computer,
and confirm that $\MMM_{\CCC}$ is even and 
of discriminant $2^{20}$.
For $P\ne P\sprime$, 
the vector
$4(e_P-e_{P\sprime})\in \MMM_{\CCC}$ 
in~\eqref{eq:combi} with  $k=1$ has norm $8$.
Hence all we have to prove is that
every non-zero vector of $\MMM_{\CCC}$ is of norm $\ge 8$.
We assume that a non-zero vector ${w_s} \in \MMM_{\CCC}$ satisfies ${w_s}^2<8$,
and derive a contradiction.
We express ${w_s}$ as a vector in $\Z^{f(4)}=\Z^{85}$ by~\eqref{eq:sgnnMMM}.
Recall that ${\bar{w}_s}\in \CCC\subset R^{85}$ is the code word ${w_s}\bmod 8$.
\par
\medskip
For $\nu=0,1,2,3$,
we put
$$
\KKK_\nu:=\Ker (\;\CCC\inj \MMM\dual\tensor R\to \MMM\dual\tensor \Z/2^\nu \Z\;),
$$
where $\MMM\dual\tensor R\to \MMM\dual\tensor \Z/2^\nu\Z$ is the reduction homomorphism.
Then we have a filtration
$$
\CCC=\KKK_0\;\supset\; \KKK_1\;\supset\; \KKK_2\;\supset\; \KKK_3=0,
$$
and each  quotient
$$
\Gamma_\nu:=\KKK_\nu/\KKK_{\nu+1}
$$ 
is naturally regarded as an $\F_2$-code of length $85$.
\par
\smallskip
We fix  terminologies.
Let $\Gamma\subset \F_2^N$ be an $\F_2$-code.
The \emph{Hamming weight} $\weight (\omega)$ of a code word 
$\omega\in \Gamma$ is the number of $1$  
that occurs in the components of $\omega$.
The \emph{weight enumerator} of  $\Gamma$ is 
the polynomial $\sum_{\omega\in \Gamma} x^{\weight(\omega)}$.
\par
\smallskip
We compute the weight enumerator of
the $\F_2$-code 
$$
\Gamma_0=\MMM_{\CCC}/(\MMM_{\CCC}\cap 2\MMM\dual) \;\subset\;
 \MMM\dual/2\MMM\dual=
\MMM\dual\tensor \F_2
$$
 of dimension $16$ by a computer.
 The result is  
$$
1+3570\,{x}^{32} + 38080\,{x}^{40} +23800\,{x}^{48}+85\,{x}^{64}.
$$
If the image of ${\bar{w}_s}\in \CCC$ by the projection $\CCC\to \Gamma_0$
were non-zero,
then ${w_s}$ would have at least $32$ odd components and hence ${w_s}^2\ge 8$.
Thus we have ${\bar{w}_s} \in \KKK_1$. 
The $\F_2$-code 
$$
\Gamma_1=(\MMM_{\CCC}\cap 2\MMM\dual)/(\MMM_{\CCC}\cap 4\MMM\dual) \;\subset\; 
2\MMM\dual/4\MMM\dual\cong
\MMM\dual\tensor \F_2
$$
is of dimension $60$.
The weight enumerator of $\Gamma_1$ cannot be calculated directly,
because $2^{60}$ is too large.
However,
the orthogonal complement $\Gamma_1\sperp$ 
of $\Gamma_1$
with respect to the standard inner product on $\F_2^{85}$ is of dimension $25$,
and hence its weight enumerator is calculated in a naive method. 
Via the MacWilliams Theorem (see~\cite[Ch.5]{MR0465510}),
we see that   the weight enumerator of $\Gamma_1$ is 
\begin{equation*}
\begin{split}
1+17850\,{x}^{8}+45696\,{x}^{10}+&8020600\,{x}^{12}+229785600\,{x}^{14}
+4668633585\,{x}^{16}+\cdots \\
&
\cdots+1142400\,{x}^{74}+23800\,{x}^{76}+357\,{x}^{80}.
\end{split}
\end{equation*}
In particular, every non-zero code word of $\Gamma_1$ is 
of Hamming weight $\ge 8$.
Therefore, if the image of ${\bar{w}_s}\in \KKK_1$ in $\Gamma_1$ were non-zero,
then ${w_s}$ would have  at least $8$ components that are congruent to $2$ modulo $4$,
and hence ${w_s}^2\ge 8$.
Thus we have ${\bar{w}_s}\in \KKK_2$.
The $\F_2$-code 
$$
\Gamma_2=(\MMM_{\CCC}\cap 4\MMM\dual)/(\MMM_{\CCC}\cap 8\MMM\dual) \;\subset\; 
4\MMM\dual/8\MMM\dual\cong
\MMM\dual\tensor \F_2
$$
is of dimension $84$,
and is defined in $\MMM\dual\tensor \F_2$  by an equation
$$
x_0+\cdots + x_{84}=0.
$$
Therefore, if the image of ${\bar{w}_s}\in \KKK_2$ in $\Gamma_2$ were non-zero,
then ${w_s}$ would have  at least $2$ components that are congruent to $4$ modulo $8$,
and hence ${w_s}^2\ge 8$.
Thus we have ${\bar{w}_s}\in \KKK_3$.
Hence every component of ${w_s}$ is congruent to $0$ modulo $8$.
Since ${w_s}$ is non-zero,  we have  ${w_s}^2\ge 8$, which contradicts the hypothesis.
\end{proof}
\begin{proof}[Proof of Theorem~\ref{thm:Nprim}]
Since $\NSigma(X)$ is generated by the vectors~\eqref{eq:combi},
we can calculate the Gram matrix of $\NSigma(X)$,
and show that $\disc \NSigma(X)=85\cdot 2^{16}$.
On the other hand,
using~\eqref{eq:SigmaSigma},~\eqref{eq:hSigma} and~\eqref{eq:hh},
we can realize $\HHH$ and $\NNN(X)$ as  submodules of  
$(\HHH+\MMM_0)\dual$
in terms of the dual basis of the basis~\eqref{eq:blist} of $\HHH+\MMM_0$,
and 
compute the Gram matrix of 
$\Nprim(X)=\HHH\sperp$.
It turns out 
that $\Nprim(X)$ is
also of discriminant $85\cdot 2^{16}$.
Hence
we conclude that $\NSigma(X)=\Nprim(X)$.
It is easy to see that
the minimal norm of $\NSigma(X)$ is  $\le 8$.
Since $\NSigma(X)$ 
 is embedded in the lattice $\MMM_{\CCC}$,
we see  that $\Nprim(X)$  is even
and of minimal norm $\ge 8$ by Theorem~\ref{thm:MCCC}.
\end{proof}
\par
\medskip
\begin{remark}
The intersection pairing of algebraic cycles 
on an algebraic variety 
in positive characteristic
has been used to construct  dense lattices.
For example, Elkies~\cite{MR1289579, MR1437492, MR1917431} and Shioda~\cite{MR1129298} constructed 
many  lattices of high density 
as Mordell-Weil lattices of elliptic surfaces in positive characteristics.
See also~\cite[page xviii]{MR1662447}.
\end{remark}
\begin{remark}
In~\cite{MR1794260},
we have obtained  a dense lattice of rank $86$ 
from  the Fermat cubic $6$-fold in characteristic $2$.
This lattice is also closely related to $\MMM_{\CCC}$.
\end{remark}
%
%
%
%
%

\section{Acknowledgement} 
Thanks are due to Professor Nobuyoshi Takahashi and Professor Shigeru Mukai for helpful discussions.
This work is
partially supported by
 JSPS Core-to-Core Program No.~18005  and 
 JSPS Grants-in-Aid for Scientific Research (B) No.~20340002.
 We also thank the referee for valuable comments and suggestions.

\bibliographystyle{plain}

\begin{thebibliography}{10}

\bibitem{MR0371899}
M.~Artin.
\newblock Supersingular {$K3$} surfaces.
\newblock {\em Ann. Sci. \'Ecole Norm. Sup. (4)}, 7:543--567 (1975), 1974.

\bibitem{MR747534}
W.~M. Beynon and N.~Spaltenstein.
\newblock Green functions of finite {C}hevalley groups of type {$E_{n}$}
  {$(n=6,\,7,\,8)$}.
\newblock {\em J. Algebra}, 88(2):584--614, 1984.

\bibitem{MR0306130}
J.~W.~S. Cassels.
\newblock {\em An introduction to the geometry of numbers}.
\newblock Springer-Verlag, Berlin, 1971.
\newblock Second printing, corrected, Die Grundlehren der mathematischen
  Wissenschaften, Band 99.

\bibitem{MR1662447}
J.~H. Conway and N.~J.~A. Sloane.
\newblock {\em Sphere packings, lattices and groups}, volume 290 of {\em
  Grundlehren der Mathematischen Wissenschaften}.
\newblock Springer-Verlag, New York, third edition, 1999.

\bibitem{MR601520}
P.~Deligne.
\newblock La conjecture de {W}eil. {II}.
\newblock {\em Inst. Hautes \'Etudes Sci. Publ. Math.}, (52):137--252, 1980.

\bibitem{MR1935564}
I.~Dolgachev and S.~Kond{\=o}.
\newblock A supersingular {$K3$} surface in characteristic 2 and the {L}eech
  lattice.
\newblock {\em Int. Math. Res. Not.}, (1):1--23, 2003.

\bibitem{MR1289579}
N.~D. Elkies.
\newblock Mordell-{W}eil lattices in characteristic {$2$}. {I}. {C}onstruction
  and first properties.
\newblock {\em Internat. Math. Res. Notices}, (8):343 ff., approx.\ 18 pp.\
  (electronic), 1994.

\bibitem{MR1437492}
N.~D. Elkies.
\newblock Mordell-{W}eil lattices in characteristic 2. {II}. {T}he {L}eech
  lattice as a {M}ordell-{W}eil lattice.
\newblock {\em Invent. Math.}, 128(1):1--8, 1997.

\bibitem{MR1917431}
N.~D. Elkies.
\newblock Mordell-{W}eil lattices in characteristic 2. {III}. {A}
  {M}ordell-{W}eil lattice of rank 128.
\newblock {\em Experiment. Math.}, 10(3):467--473, 2001.

\bibitem{MR1644323}
W.~Fulton.
\newblock {\em Intersection theory}, volume~2 of {\em Ergebnisse der Mathematik
  und ihrer Grenzgebiete. 3. Folge. A Series of Modern Surveys in Mathematics}.
\newblock Springer-Verlag, Berlin, second edition, 1998.

\bibitem{MR1288523}
P.~Griffiths and J.~Harris.
\newblock {\em Principles of algebraic geometry}.
\newblock Wiley Classics Library. John Wiley \& Sons Inc., New York, 1994.
\newblock Reprint of the 1978 original.


\bibitem{MR0463157}
R.~Hartshorne.
\newblock {\em Algebraic geometry}.
\newblock Springer-Verlag, New York, 1977.
\newblock Graduate Texts in Mathematics, No. 52.


\bibitem{MR1247497}
W.~Haboush and N.~Lauritzen.
\newblock Varieties of unseparated flags.
\newblock In {\em Linear algebraic groups and their representations ({L}os
  {A}ngeles, {CA}, 1992)}, volume 153 of {\em Contemp. Math.}, pages 35--57.
  Amer. Math. Soc., Providence, RI, 1993.


\bibitem{MR1186416}
Johan~P. Hansen.
\newblock Deligne-{L}usztig varieties and group codes.
\newblock In {\em Coding theory and algebraic geometry ({L}uminy, 1991)},
  volume 1518 of {\em Lecture Notes in Math.}, pages 63--81. Springer, Berlin,
  1992.

\bibitem{MR1866342}
S{\o}ren~Have Hansen.
\newblock Error-correcting codes from higher-dimensional varieties.
\newblock {\em Finite Fields Appl.}, 7(4):531--552, 2001.

\bibitem{MR1183809}
N.~Lauritzen.
\newblock {T}he {E}uler characteristic of a homogeneous line bundle.
\newblock {\em C. R. Acad. Sci. Paris S\'er. I Math.}, 315(6):715--718, 1992.

\bibitem{MR1385284}
N.~Lauritzen.
\newblock Embeddings of homogeneous spaces in prime characteristics.
\newblock {\em Amer. J. Math.}, 118(2):377--387, 1996.

\erase{
\bibitem{MR1473603}
N.~Lauritzen.
\newblock Schubert cycles, differential forms and {$\mathcal D$}-modules on
  varieties of unseparated flags.
\newblock {\em Compositio Math.}, 109(1):1--12, 1997.
}


\bibitem{MR0465510}
F.~J. MacWilliams and N.~J.~A. Sloane.
\newblock {\em The theory of error-correcting codes.}
\newblock {N}orth-Holland Publishing Co., Amsterdam, 1977.
\newblock {N}orth-Holland Mathematical Library, Vol. 16.

\bibitem{MR0472817}
W.~Messing.
\newblock Short sketch of {D}eligne's proof of the hard {L}efschetz theorem.
\newblock In {\em Algebraic geometry ({P}roc. {S}ympos. {P}ure {M}ath., {V}ol.
  29, {H}umboldt {S}tate {U}niv., {A}rcata, {C}alif., 1974)}, pages 563--580.
  Amer. Math. Soc., Providence, R.I., 1975.

\bibitem{MR1759842}
F.~Rodier.
\newblock Nombre de points des surfaces de {D}eligne et {L}usztig.
\newblock {\em J. Algebra}, 227(2):706--766, 2000.

\bibitem{MR633161}
A.~N. Rudakov and I.~R. Shafarevich.
\newblock Surfaces of type {$K3$} over fields of finite characteristic.
\newblock In {\em Current problems in mathematics, Vol. 18}, pages 115--207.
  Akad. Nauk SSSR, Vsesoyuz. Inst. Nauchn. i Tekhn. Informatsii, Moscow, 1981.
\newblock Reprinted in I. R. Shafarevich, Collected Mathematical Papers,
  Springer-Verlag, Berlin, 1989, pp. 657--714.

\bibitem{MR1176080}
I.~Shimada.
\newblock Unirationality of certain complete intersections in positive
  characteristics.
\newblock {\em T\^ohoku Math. J. (2)}, 44(3):379--393, 1992.

\bibitem{MR1794260}
I.~Shimada.
\newblock Lattices of algebraic cycles on {F}ermat varieties in positive
  characteristics.
\newblock {\em Proc. London Math. Soc. (3)}, 82(1):131--172, 2001.

\bibitem{MR0374149}
T.~Shioda.
\newblock An example of unirational surfaces in characteristic {$p$}.
\newblock {\em Math. Ann.}, 211:233--236, 1974.

\bibitem{MR1129298}
T.~Shioda.
\newblock Mordell-{W}eil lattices and sphere packings.
\newblock {\em Amer. J. Math.}, 113(5):931--948, 1991.

\bibitem{MR526513}
T.~Shioda and T.~Katsura.
\newblock On {F}ermat varieties.
\newblock {\em T\^ohoku Math. J. (2)}, 31(1):97--115, 1979.

\bibitem{MR723216}
T.~Shoji.
\newblock On the {G}reen polynomials of classical groups.
\newblock {\em Invent. Math.}, 74(2):239--267, 1983.

\bibitem{MR1096262}
C.~Wenzel.
\newblock Classification of all parabolic subgroup-schemes of a reductive
  linear algebraic group over an algebraically closed field.
\newblock {\em Trans. Amer. Math. Soc.}, 337(1):211--218, 1993.

\erase{
\bibitem{MR1123669}
C.~Wenzel.
\newblock Rationality of {$G/P$} for a nonreduced parabolic subgroup-scheme
  {$P$}.
\newblock {\em Proc. Amer. Math. Soc.}, 117(4):899--904, 1993.
}

\end{thebibliography}

\def\cprime{$'$} \def\cprime{$'$} \def\cprime{$'$} \def\cprime{$'$}

\end{document}